\documentclass[11pt,reqno]{amsproc}
\title{On well-posedness of the Muskat problem with surface tension}

\author{Huy Q. Nguyen}
\address{Department of Mathematics, Brown University, Providence, RI 02912}
\email{hnguyen@math.brown.edu}

\usepackage[margin=1in]{geometry}
\usepackage{amsmath, amsthm, amssymb, mathrsfs, stmaryrd}
\usepackage{times}
\usepackage{color}
\usepackage{hyperref}
\usepackage{graphicx}
\usepackage{subcaption}
\usepackage{float}

\newcommand{\bq}{\begin{equation}}
\newcommand{\eq}{\end{equation}}
\newcommand{\bqa}{\begin{eqnarray*}}
\newcommand{\eqa}{\end{eqnarray*}}

\theoremstyle{plain}
\newtheorem{theo}{Theorem}[section]
\newtheorem{prop}[theo]{Proposition}
\newtheorem{lemm}[theo]{Lemma}

\newtheorem{defi}[theo]{Definition}
\theoremstyle{definition}
\newtheorem{rema}[theo]{Remark}
\newtheorem{nota}[theo]{Notation}
\DeclareMathOperator{\cnx}{div}

\DeclareMathOperator{\dist}{dist}

\DeclareSymbolFont{pletters}{OT1}{cmr}{m}{sl}
\DeclareMathSymbol{s}{\mathalpha}{pletters}{`s}

\def\tt{\theta}
\def\eps{\varepsilon}
\def\na{\nabla}
\def\la{\left\lvert}

\def\ra{\right\vert}
 
\def\lb{\llbracket}
\def\rb{\rrbracket}
\def\les{\lesssim}

\def\mez{\frac{1}{2}}
\def\tdm{\frac{3}{2}}

\def\Rr{\mathbb{R}}
\def\T{\mathbb{T}}
\def\Nn{\mathbb{N}}
\def\Cc{\mathbb{C}}

\def\cF{\mathcal{F}}

\def\cR{\mathcal{R}}
\def\ld{\lambda}

\def\p{\partial}
\def\na{\nabla}

\def\wt{\widetilde}

\def\g{\mathfrak{g}}
\def\s{\mathfrak{s}}

\numberwithin{equation}{section}

\date{today}
\begin{document}
\begin{abstract}
We consider the Muskat problem with surface tension for one fluid  or two fluids, with or without viscosity jump, with infinite depth or Lipschitz rigid boundaries, and in arbitrary dimension $d$ of the interface. The problem  is nonlocal, quasilinear, and to leading order,  is scaling invariant in the Sobolev space $H^{s_c}(\Rr^d)$ with $s_c=1+\frac d2$. We prove local well-posedness  for {\it large data} in {\it all subcritical Sobolev spaces} $H^s(\Rr^d)$, $s>s_c$, allowing for initial interfaces whose curvatures are unbounded and, furthermore when $d=1$,  not locally square integrable. To the best of our  knowledge, this is the first large-data well-posedness result that covers all subcritical Sobolev spaces for the Muskat problem with surface tension. We  reformulate the problem in terms of the Dirichlet-Neumann operator and use a paradifferential approach to reduce the problem to an explicit parabolic equation, which is of independent interest. 
\end{abstract}


\noindent

\maketitle


\section{Introduction}

\subsection{The Muskat problem} The  Muskat problem (\cite{Mus}) of practical importance in geoscience describes the dynamics of two immiscible fluids  in a porous medium with different densities $\rho^\pm$ and different viscosities $\mu^\pm$. Let us denote the interface between the two fluids by $\Sigma$ and assume that it is the graph of a time-dependent function $\eta(x, t)$
\bq\label{Sigma}
\Sigma_t=\{(x, \eta(t,x)): x\in \Rr^d\},
\eq
where $d\ge 1$ is the horizontal dimension.  The associated time-dependent fluid domains are then given by
\begin{align}\label{Omega+}
&\Omega^+_t=\{(x, y)\in \Rr^{d}\times \Rr: \eta(t, x)<y<\underline{b}^+(x)\},\\
\label{Omega-}
&\Omega^-_t=\{(x, y)\in \Rr^{d}\times \Rr: \underline{b}^-(x)<y<\eta(t, x)\},
\end{align}
where $\underline{b}^\pm$ are the parametrizations of the rigid boundaries
\bq\label{Gamma:pm}
\Gamma^\pm =\{(x, \underline{b}^\pm(x)): x\in \Rr^d\}.
\eq
Here $x$ is the horizontal variable and $y$ is the vertical variable.

The incompressible fluid velocity $u^\pm$ in each region is governed by Darcy's law
\begin{align}\label{Darcy:pm}
&\mu^\pm u^\pm+\nabla_{x, y}p^\pm=-\rho^\pm \g \vec{e}_{d+1},\quad \cnx_{x, y} u^\pm =0\quad \text{in}~\Omega^\pm_t,
\end{align}
where $\g$ is the acceleration due to gravity and $\vec{e}_{d+1}$ is the $(d+1)$th vector of the canonical basis of $\Rr^{d+1}$.

At the interface $\Sigma$, the normal velocity is continuous
\bq\label{u.n:pm}
u^+\cdot n=u^-\cdot n\quad\text{on}~\Sigma_t
\eq
where $n=\frac{1}{\sqrt{1+|\nabla \eta|}}(-\nabla \eta, 1)$ is the upward pointing unit normal to $\Sigma_t$. Then, the interface moves with the fluid
\bq\label{kbc:pm}
\p_t\eta=\sqrt{1+|\nabla \eta|^2}u^-\cdot n\vert_{\Sigma_t}.
\eq
According to the Young-Laplace equation, the pressure jump at the interface is proportional to the mean curvature $H(\eta)$
\bq\label{p:pm}
p^--p^+= \s H(\eta):=-\s\cnx\Big(\frac{\na\eta}{\sqrt{1+|\na\eta|^2}}\Big)\quad\text{on}~\Sigma_t,
\eq
where $\s \ge 0$ denotes the surface tension coefficient.

Finally, at the two rigid boundaries, the no-penetration boundary conditions are imposed
\bq\label{bc:b:pm}
u^\pm\cdot \nu^\pm=0\quad \text{on}~\Gamma^\pm,
\eq
where $\nu^\pm=\pm\frac{1}{\sqrt{1+|\nabla \underline b^\pm|}}(-\nabla \underline b^\pm, 1)$ denotes the outward pointing unit normal to $\Gamma^\pm$. We will also consider the case that at least one of $\Gamma^\pm$ is empty (infinite depth)  in which case the velocity $u$ vanishes at infinity.

We shall refer to the system \eqref{Darcy:pm}-\eqref{bc:b:pm} as {\it the two-phase Muskat problem}. When the top phase corresponds to vacuum, i.e. $\mu^+=\rho^+=0$, the two-phase Muskat problem reduces to {\it the one-phase Muskat problem} and \eqref{p:pm} becomes
\bq\label{bcp:1p}
p^-=\s H(\eta)\quad\text{on}~\Sigma_t.
\eq
We note that the Muskat problem is mathematically analogous to the vertical Hele-Shaw problem with gravity \cite{Hel1, Hel2}.
\subsection{Reformulation and main results}
Our reformulation for the Muskat problem involves the Dirichlet-Neumann operators $G^\pm(\eta)$ associated to $\Omega^\pm$. For a given function $f$, letting $\phi^\pm$ solve 
\bq\label{elliptic:intro}
\begin{cases}
\Delta_{x, y} \phi^\pm=0\quad\text{in}~\Omega^\pm,\\
\phi^\pm=f\quad\text{on}~\Sigma,\\
\frac{\p\phi^\pm}{\p \nu^\pm}=0\quad\text{on}~\Gamma^\pm,
\end{cases}
\eq
we define
\bq\label{defG:intro}
G(\eta)^\pm f=\sqrt{1+|\nabla \eta|^2}\frac{\p\phi^\pm}{\p n}.
\eq
\begin{prop}[{\bf Reformulation}]\label{prop:reform} 
$(i)$ If $(u, p, \eta)$ solve the one-phase Muskat problem then $\eta:\Rr^d\to \Rr$  obeys the equation 
\bq
\p_t\eta=-\frac{1}{\mu^-} G^{-}(\eta)\big(\s H(\eta)+\rho^-\g\eta\big).\label{eq:eta}
\eq
Conversely, if $\eta$  is a solution of \eqref{eq:eta} then the one-phase Muskat problem has a solution which admits $\eta$ as the free surface.

$(ii)$ If $(u^\pm, p^\pm, \eta)$ is a solution of the two-phase Muskat problem then 
\bq\label{eq:eta2p}
\p_t\eta=-\frac{1}{\mu^-}G^{-}(\eta)f^-,
\eq
where $f^\pm:=p^\pm\vert_{\Sigma}+\rho^\pm\g \eta$ satisfy
\bq\label{system:fpm}
\begin{cases}
 f^--f^+=\s H(\eta)+\lb \rho\rb\g\eta,\quad \lb \rho\rb=\rho^--\rho^+,\\
\frac{1}{\mu^+}G^+(\eta)f^+-\frac{1}{\mu^-}G^-(\eta)f^-=0.
\end{cases}
\eq
Conversely, if $\eta$ is a solution of  \eqref{eq:eta2p} where $f^\pm$ solve \eqref{system:fpm} then the two-phase Muskat problem has a solution which admits $\eta$ as the free interface.
\end{prop}
We postpone the proof of Proposition \ref{prop:reform} to Appendix \ref{appendix:reform}. The above reformulation contains as a special case the reformulation obtained in \cite{NgPa} in the absence of surface tension, i.e. $\s=0$. In this work, we are interested in the case that $\s$ is a fixed positive constant. To leading order, since $\s H(\eta)+\rho^-\g\eta\sim -\s\Delta_x \eta$, equation \eqref{eq:eta} behaves like
\bq\label{leading1p}
\p_t\eta=-\frac{\s}{\mu^-}G^-(\eta)\Delta_x\eta.
\eq
It can be easily checked that in the case of no bottoms $\Gamma^\pm=\emptyset$, if  $\eta(t, x)$ solves \eqref{leading1p} then so is 
\[
\eta_\ld(t, x)=\ld^{-1}\eta(\ld^3 t, \ld x)\quad\forall \ld>0,
\]
and thus the ($L^2$-based) Sobolev space $\dot H^{1+\frac d2}(\Rr^d)$ is scaling invariant. Interestingly, the Muskat problem without surface tension (and without bottoms) also admits $\dot H^{1+\frac d2}(\Rr^d)$ as the scaling invariant Sobolev space (\cite{NgPa}). The presence of bottoms alters the behavior of solutions at low frequencies. Our main results state that the Muskat problem with surface tension is locally well-posed  for {\it large data} in {\it all subcritical Sobolev spaces} $H^s(\Rr^d)$, $s>1+\frac d2$, either for one fluid or two fluids, with or without viscosity jump, with infinite depth or with Lipschitz rigid boundaries, and in arbitrary dimension. Here well-posedness is obtained in the  sense of Hadamard: existence, uniqueness and Lipschitz dependence on initial data.

Introducing  the spaces
\begin{align*}
&\dot W^{1, \infty}(\Rr^d)=\big\{v\in L^1_{loc}(\Rr^d): \na v\in L^\infty(\Rr^d)\big\},\\
&Z^s(T)=C([0, T]; H^s(\Rr^d))\cap L^2([0, T]; H^{s+\tdm}(\Rr^d)),
\end{align*}
we  state our main results in the following theorems. 
\begin{theo}[{\bf Well-posedness for the one-phase problem}]\label{theo:1p}
Let $\mu^->0$, $\rho^->0$ and $\s>0$. Let $s>1+\frac d2$ be a real number with $d\ge 1$. Consider either $\Gamma^-=\emptyset$ or $\underline{b}^{-}\in \dot W^{1, \infty}(\Rr^d)$. Let $\eta_0\in H^s(\Rr^d)$ satisfy 
\[
\dist (\eta_0, \Gamma^-)>2h>0.
\]
 Then there exist $T>0$, depending only on $\| \eta_0\|_{H^s}$ and $(h, s, \frac{\rho^-\g}{\mu^-}, \frac{\s}{\mu^-})$,  and a unique solution $\eta\in Z^s(T)$ to \eqref{eq:eta} such that $\eta\vert_{t=0}=\eta_0$ and
\[
\inf_{t\in [0, T]}\dist(\eta(t),\Gamma^-)>h.
\]
Moreover, if $\eta_1$ and $\eta_2$ are two solutions of \eqref{eq:eta} then the stability estimate 
\bq
\| \eta_1-\eta_2\|_{Z^s(T)}\le \cF\big(\|(\eta_1, \eta_2)\|_{Z^s(T)}\big)\| (\eta_1-\eta_2)\vert_{t=0}\|_{H^s}
\eq
holds for some function $\cF:\Rr^+\to \Rr^+$ depending only on $ (h, s, \frac{\rho^-\g}{\mu^-}, \frac{\s}{\mu^-})$.
\end{theo}
\begin{theo}[{\bf Well-posedness for the two-phase problem}]\label{theo:2p}
Let $\mu^\pm >0$, $\rho^\pm >0$ and $\s>0$. Let $s>1+\frac d2$ be a real number  with $d\ge 1$. Consider any combination of $\Gamma^\pm=\emptyset$ and $\underline{b}^\pm\in \dot W^{1, \infty}(\Rr^d)$. Let $\eta_0\in H^s(\Rr^d)$ satisfy 
\[
\dist (\eta_0, \Gamma^\pm)>2h>0.
\]
Then there exist $T>0$, depending only on $\| \eta_0\|_{H^s}$ and $(h, s, \s, \mu^\pm, \lb \rho\rb\g)$,  and a unique solution $\eta\in Z^s(T)$ to \eqref{eq:eta2p}-\eqref{system:fpm} such that $\eta\vert_{t=0}=\eta_0$ and
\[
\inf_{t\in [0, T]}\dist(\eta(t),\Gamma^\pm)>h.
\]
Moreover,  if $\eta_1$ and $\eta_2$ are two solutions of \eqref{eq:eta2p}-\eqref{system:fpm} then the stability estimate 
\bq
\| \eta_1-\eta_2\|_{Z^s(T)}\le \cF\big(\|(\eta_1, \eta_2)\|_{Z^s(T)}\big)\| (\eta_1-\eta_2)\vert_{t=0}\|_{H^s}
\eq
holds for some function $\cF:\Rr^+\to \Rr^+$ depending only on $(h, s, \s, \mu^\pm, \lb \rho\rb\g)$.
\end{theo}
To the best of our  knowledge, Theorems \ref{theo:1p} and \ref{theo:2p} are the first  large-data well-posedness results  that cover all subcritical Sobolev spaces for the Muskat problem with surface tension.  The corresponding results in the absence of surface tension were obtained in the recent work \cite{NgPa}; see Subsection \ref{section:prior}  for a discussion on prior results. In particular, Theorems \ref{theo:1p} and \ref{theo:2p} allow for initial interfaces whose  curvatures are {\it unbounded} for $d\ge 1$ and  {\it not locally square integrable} for $d=1$. 

Using results on paralinearization of the Dirichlet-Neumann operator obtained in \cite{ABZ3, NgPa} we shall reduce both the one-phase and two-phase Muskat problems with surface tension to the following explicit {\it parabolic paradifferential equation}
\bq\label{reduce:intro}
\p_t \eta+\frac{\s}{\mu^++\mu^-}T_{\ld \ell} \eta=g
\eq
where $g$ satisfies 
\bq\label{estg:intro}
\| g\|_{H^{s-\tdm+\delta}}\le \cF(\| \eta\|_{H^s})\| \eta\|_{H^{s+\tdm}}
\eq
provided that 
\bq
s>1+\frac d2\quad\quad\text{and} \quad  \delta \in (0, s-1-\frac d2)~\text{and}~ \delta\le \mez.
\eq
We refer to Propositions \ref{prop:reduce} and \ref{prop:reduce:2p} for the precise statements and to Appendix \ref{appendix:para} for notation of paradifferential operators. Here $\ld(x, \xi)$ and $\ell(x, \xi)$, defined by  \eqref{ld} and \eqref{def:ell}, are respectively the principal symbol of the Dirichlet-Neumann operator $G^-(\eta)$ and the mean curvature operator $H(\eta)$; moreover they are elliptic and  of first and second order respectively. Consequently, $T_{\ld \ell}$ is an elliptic paradifferential operator of third order and thus the solution $\eta$ to \eqref{reduce:intro} gains $\tdm$ derivatives when measured in $L^2_t$. The estimate \eqref{estg:intro} then shows that for any subcritical data  $\eta_0\in H^s(\Rr^d)$,  the right-hand side $g$ is smoothing which in turn allows one to close the energy estimate in $L^\infty_t H^s_x\cap L^2_t H^{s+\tdm}_x$. The stability estimate is more delicate, especially for the two-phase problem.

 The reduction \eqref{reduce:intro}-\eqref{estg:intro} is of independent interest. It is worth remarking that unlike the case of zero surface tension \cite{NgPa}, this reduction does not involve  the trace of velocity on the interface.

This work emphasizes the virtue of the  paradifferential calculus approach in establishing (almost) sharp  large-date well-posedness for free boundary problems in fluid dynamics. In the context of water waves, this approach was initiated in \cite{ABZ1, ABZ3, ABZ4} with inspiration from \cite{AlaGuy, LannesJAMS}. In the context of Muskat, this approach was independently employed in \cite{AL, NgPa} for the case without surface tension. In this work, by taking advantage of the strong dissipation mechanism of the Muskat problem with surface tension, we obtain  well-posedness results that allow for  curvature singularity of initial data. Such a result for the water waves problem with surface tension remains open in view of the recent works \cite{ABZ1, PoyNg1, PoyNg2, Ng}.
 \begin{rema}
Theorems \ref{theo:1p} and \ref{theo:2p} still hold in the following situations:
\begin{itemize}
\item  Gravity is neglected ($\g=0$), as usually assumed for the Hele-Shaw problem.
\item  Periodic data $\eta_0\in H^s(\T^d)$ for any $s>1+\frac d2$.
\end{itemize}
Since $\underline{b}^\pm\in \dot W^{1, \infty}(\Rr^d)$, the rigid boundaries $\Gamma^\pm$ can be unbounded. The proof of Theorem \ref{theo:2p}  (see \eqref{est:fpm}) also gives that 
\bq
f^\pm \in L^\infty([0, T]; \wt H^{s-2}_\pm(\Rr^d))\cap L^2([0, T]; \wt H^{s-\mez}_\pm(\Rr^d)).
\eq
For quasilinear PDEs, stability estimates for solutions are usually obtained in rougher topology compared to initial data, e.g. \cite{ABZ1, ABZ3, NgPa}. Theorems \ref{theo:1p} and \ref{theo:2p} however provide stability estimates for solutions in the same topology as initial data.
\end{rema}
\begin{rema}
It is well known that the smoothing effect of surface tension bypasses the Rayleigh-Taylor stability condition required for well-posedness of free boundary problems in the absence of surface tension. In particular, Theorem \ref{theo:2p} {\it does not} require that the less dense fluid is above the denser one, i.e. $\rho^+<\rho^-$. We refer to \cite{CasCorFefGanMar, CheGraShk,  LannesJAMS, NgPa, Wu2d, Wu3d} for in-depth discussions on the Rayleigh-Taylor stability condition for Muskat and water waves.
\end{rema}
\begin{rema}
For simplicity let us consider the infinite-depth case and restrict ourselves to the graph formulation. As a consequence of the fact that the existence time $T$ in Theorems \ref{theo:1p} and \ref{theo:2p} depends only on the $H^s(\Rr^d)$ norm of initial data, if $\| \eta(t)\|_{ H^s}$ remains uniformly bounded in $t$  up to time $T$ then the solution can be continued past $T$. It is possible  that by combining the techniques in the present paper with mixed H\"older-Sobolev  estimates for the Dirichlet-Neumann operator in the spirit of \cite{ABZ4, PoyNg1}, one can prove that  controlling $\| \eta\|_{L^\infty((0, T); W^{1+\eps, \infty})}$ for any small $\eps>0$ would suffice. 

It is an open problem for the Muskat problem (with or without surface tension) whether the control of the maximal slope $\| \na_x\eta\|_{L^\infty((0, T); L^\infty)}$ implies the continuation of the solution. Any continuation criterion in terms of scaling invariant quantities should be interesting. For the 2D Muskat problem without surface tension and  constant viscosity, it is known from  \cite{ConGanShvVic} that the solution remains regular so long as the slope $\p_x\eta$ remains bounded and uniformly continuous.
 \end{rema}
\begin{rema}
The time interval  $[0, T]$ in Theorems \ref{theo:1p} and \ref{theo:2p} shrinks to $0$ as the surface tension coefficient $\s$ vanishes. The question of zero surface tension limit is interesting but is not be pursued in the present paper. We refer to  \cite{Amb-stl, Ambrose3D, FlynNgu} for  results in this direction.
\end{rema}
\subsection{Priori results}\label{section:prior}
The Muskat problem and its mathematical analog -- the Hele-Shaw problem have recently been the subject of intense study in analysis of PDEs and numerical analysis.  The literature is vast and we will mostly discuss the topic of well-posedness. We refer to the recent surveys  \cite{GanSur, GraLaz} for discussions on other topics, and in particular  \cite{CasCorFefGan, CasCorFefGanMar,GomGra} for interesting results on finite-time singularity formation.  

 Taking advantage of the parabolic nature of the Muskat problem, global strong solutions for {\it small data} have been considered in a large number of studies. We refer to \cite{ Cam, Chen, CheGraShk, ConCorGanRodPStr, ConCorGanStr, ConGanShvVic, ConPug, CorGan07, CorLaz, SieCafHow} for data in subcritical $L^2$-based and $L^\infty$-based Sobolev spaces, and to \cite{GanGarPatStr} for data in the critical   Wiener space $\dot\cF^{1, 1}$. We note in particular that  \cite{CheGraShk, GanGarPatStr} allow for viscosity jump and \cite{CorLaz} allows for interfaces with large slopes. In the case of constant viscosity, by using maximum principles for the slope, global weak solutions were constructed in \cite{ConCorGanStr, DenLeiLin}.

 We discuss in detail the issue of local well-posedness for {\it large data}. In the context of the Musat problem, the case {\it without surface tension} is better understood.  Early results on  local well-posedness for large data in Sobolev spaces  date back to \cite{Chen, EscSim, Yi, Amb0,  Amb}. C\'ordoba and Gancedo \cite{CorGan07} introduced  the contour dynamics formulation for the Muskat problem without viscosity jump and with infinite depth, and proved local well-posedness in $H^3(\Rr)$ and $H^4(\Rr^2)$ when the interface is a graph. In \cite{CorCorGan, CorCorGan2}, C\'ordoba, C\'ordoba and Gancedo  extended this result to the case of viscosity jump and nongraph interfaces satisfying the arc-chord and the Rayleigh-Taylor conditions.  One of the main difficulties is  to invert a highly nonlocal equation to express the vorticity amplitude in terms of the interface. Using an ``arbitrary Lagrangian-Eulerian'' approach, Cheng,  Granero and Shkoller  \cite{CheGraShk} (see also \cite{GraShk}) proved local well-posedness for the one-phase problem with flat bottoms when the initial surface $\eta\in H^2(\T)$, allowing for unbounded curvatures. This result was then extended by Matioc \cite{Mat2} to the case of viscosity jump (but infinite depth). For the case of constant viscosity, using nonlinear lower bounds, a technique developed for critical SQG, the authors in \cite{ConGanShvVic} obtained local well-posedness for $\eta\in W^{2, p}(\Rr)$ for all $p\in (1, \infty]$. The space $W^{2, 1}(\Rr)$ is scaling invariant yet requires $\mez$ more derivative compared to $H^{\tdm}(\Rr)$.  Matioc \cite{Mat} sharpened the local well-posedness theory to $\eta \in H^{\tdm+\eps}(\Rr)$ for the case of constant viscosity and infinite depth. This is the first result that covers all subcritical ($L^2$-based) Sobolev spaces for the given one-dimensional setting. By paralinearizing the nonlinearity in the contour dynamics formulation, Alazard and Lazar \cite{AL} gave a different proof and extended the result in \cite{Mat} to homogeneous Sobolev spaces, allowing non-$L^2$ solutions. In the recent joint work \cite{NgPa} of the author, we reformulated the Muskat problem in terms of the Dirichlet-Neumann operator for the general setting: one fluid or two fluids, with or without viscosity jump, with or without rigid boundaries and in arbitrary dimension.  Then employing a paradifferential calculus approach we proved local well-posedness for large data in all subcritical Sobolev spaces. In \cite{AMS}, a similar result was independently obtained for the case of one fluid and without bottom. 
 
Next we discuss results on large-data well-posedness for the Muskat and Hele-Shaw problems {\it with surface tension}, which is the problem considered in the present paper.  Early results for the 2D case date back to Duchon and Robert \cite{DucRob}, Chen \cite{Chen} and Escher-Simonett \cite{EscSim} where the initial interface is smooth enough so that its curvature is at least bounded. In \cite{Amb-stl}, the zero surface tension limit is established for the 2D Muskat problem with smooth $(H^6$) Sobolev data. The issue of low regularity well-posedness has been recently addressed for constant viscosity and viscosity jump respectively in \cite{Mat} and  \cite{Mat2} in which the initial one-dimensional interface is  taken in $H^s(\Rr)$ with $s\in (2, 3)$. These results are {\it $(\mez+\eps)$-derivative above scaling}, i.e. $H^{2+\eps}(\Rr)$ versus $H^{\tdm}(\Rr)$, yet allows for unbounded curvatures. The same result for the periodic case was obtained in \cite{Mat3}. Our  Theorems \ref{theo:1p} and \ref{theo:2p} appear to be the first large-data well-posedness results that cover {\it all subcritical Sobolev spaces} for the Muskat problem with surface tension in {\it a general setting}.

 The paper is organized as follows. In Section \ref{sec:DN}, we recall results on the continuity, paralinearization and contraction estimates for the Dirichlet-Neumann operator, most of which are taken from \cite{ABZ3} and \cite{NgPa} . Sections \ref{sec:1p} and \ref{sec:2p} are devoted to the proofs of Theorems \ref{theo:1p} and \ref{theo:2p}. Appendix \ref{appendix:para} provides a review of the paradifferential calculus machinery. Finally, we prove Proposition \ref{prop:reform} in Appendix \ref{appendix:reform}. 
 \begin{nota}
 Throughout this paper we use $\cF$ to denote a continuous increasing positive nonlinear function which may change from line to line but its dependency on relevant parameters will be indicated. 
 \end{nota}
\section{Results on the Dirichlet-Neumann operator}\label{sec:DN}
We  consider the Dirichlet-Neumann problem associated to the fluid domain $\Omega^-$ defined by \eqref{Omega-} with the time variable being frozen. We shall always assume that at least $\eta\in W^{1, \infty}(\Rr^d)$. Regarding the bottom $\Gamma^-$, we assume either $\Gamma^-=\emptyset$ or  $\Gamma^-=\{(x, \underline b^-(x)): x\in \Rr^d\}$, where $\underline b^-\in \dot W^{1, \infty}(\Rr)$  satisfying $\dist(\Sigma, \Gamma^-)>h>0$.
 Consider the elliptic problem
\bq\label{eq:elliptic}
\begin{cases}
\Delta_{x, y} \phi=0\quad\text{in}~\Omega^-,\\
\phi=f\quad\text{on}~\Sigma,\\
\frac{\p\phi}{\p \nu^-}=0\quad\text{on}~\Gamma^-,
\end{cases}
\eq
where in the case of infinite depth ($\Gamma^-=\emptyset$), the Neumann condition is replaced by the decay condition 
\[\label{vanish:dyphi}
\lim_{y\to -\infty} \na_{x, y}\phi=0.
\]
The Dirichlet-Neuman operator associated to $\Omega^-$ is formally defined by
\bq\label{def:DN}
G^-(\eta) f=\sqrt{1+|\nabla \eta|^2}\frac{\p\phi}{\p n},
\eq
where we recall that $n$ is the upward-pointing unit normal to $\Sigma$. Similarly, if $\phi$ solves the elliptic problem \eqref{eq:elliptic} with $(\Omega^-, \Gamma^-, \nu^-)$ replaced by $(\Omega^+, \Gamma^+, \nu^+)$ then we define 
\begin{equation*}
G^+(\eta) f=\sqrt{1+|\nabla \eta|^2}\frac{\p\phi}{\p n}.
\end{equation*}
Note that $n$ is inward-pointing for $\Omega^+$, making $G^+(\eta)$ a skew-adjoint operator, whereas $G^-(\eta)$ is self-adjoint. In the rest of this section, we only state results for $G^-(\eta)$ since corresponding results for  $G^+(\eta)$ are completely parallel.

The Dirichlet data $f$ for \eqref{eq:elliptic} will be taken  in the following screened fractional Sobolev space (see \cite{LeoTice})
\bq\label{def:wtHmez}
\wt H^\mez_{\Upsilon}(\Rr^d)=\left\{ f\in  \mathcal{S}'(\Rr^d)\cap  L^2_{loc}(\Rr^d): \int_{\Rr^d}\int_{B_{\Rr^d}(0, \Upsilon(x))}\frac{|f(x+x')-f(x)|^2}{|x'|^{d+1}}dx'dx<\infty\right\}/ \Rr,
\eq
where $\Upsilon:\Rr^d\to (0, \infty]$ is a given lower semi-continuous function. For the bottom domain $\Omega^-$, we will choose 
\bq\label{def:frakd}
\Upsilon(x)=
\begin{cases}
 \infty\quad\text{when}\quad \Gamma^-=\emptyset,\\
\mathfrak{d}_-(x):=\frac{\eta(x)-\underline b^-(x)}{2(\| \na_x\eta\|_{L^\infty}+\| \na_x\underline b^-\|_{L^\infty})}\quad\text{when}\quad \underline b^-\in \dot W^{1, \infty}(\Rr^d).
\end{cases}
\eq
Since $\dist(\Sigma, \Gamma^-)>h$, we have
\bq\label{lower:frakd}
\mathfrak{d}_- \ge \frac{h}{2(\| \na_x\eta\|_{L^\infty}+\| \na_x\underline b^-\|_{L^\infty})}.
\eq

We also define the slightly-homogeneous Sobolev spaces
\bq\label{def:dotHs}
 H^{1,\sigma}(\Rr^d)=\{f\in \mathcal{S}'(\Rr^d)\cap L^2_{loc}(\Rr^d): \na f\in H^{\sigma-1}(\Rr^d)\}/~\Rr.
\eq
The continuous embeddings
\bq\label{wtH:dotH}
\dot{H}^\frac{1}{2}(\mathbb{R}^d)= \wt H^\mez_\infty(\Rr^d)\subset \wt H^\mez_{\mathfrak{d}_-}(\Rr^d)\subset \wt H^\mez_1(\Rr^d) = H^{1,\frac{1}{2}}(\Rr^d)
 \eq
hold (see \cite{NgPa}). Here the embedding $\wt H^\mez_{\mathfrak{d}_-}(\Rr^d)\subset \wt H^\mez_1(\Rr^d)$ is due to the lower bound \eqref{lower:frakd}. In addition, if $\underline b^-\in W^{1, \infty}(\Rr^d)$ then according to Theorem 3.13 \cite{LeoTice}, we have  $\wt H^\mez_{\mathfrak{d}_-}(\Rr^d)=\wt H^\mez_1(\Rr^d)$. However, we have only assumed that $\underline b^-\in \dot W^{1, \infty}(\Rr^d)$ to accommodate unbounded bottoms. Nevertheless, Proposition 3.2 \cite{NgPa} implies that for any two surfaces $\eta_1$ and $\eta_2$ in $L^\infty(\Rr^d)$  satisfying $\dist(\eta_j, \underline b^-)>h>0$, the screened Sobolev space $\wt H^\mez_{\frak{d}}(\Rr^d)$, $\frak{d}$ given by \eqref{def:frakd}, is independent of $\eta_j$. This justifies the following notation.
 \begin{nota}
We denote 
\bq\label{wtH-}
\wt H^\mez_{\pm}(\Rr^d)=
\begin{cases}
\wt H^\mez_\infty(\Rr^d)~\quad \text{if}~\Gamma^\pm=\emptyset,\\
\wt H^\mez_{\mathfrak{d}_\pm}(\Rr^d)~\quad \text{if}~\underline b^\pm\in \dot W^{1, \infty}(\Rr^d),
\end{cases}
\eq
where $\mathfrak{d}_+$ is defined similarly to $\mathfrak{d}_-$ with $\underline b^-$ replaced by $\underline{b}^+$.
For $s>\mez$, we set
\begin{equation}\label{def:wtHs}
\widetilde{H}^s_\pm(\Rr^d)=\widetilde{H}^{\frac{1}{2}}_\pm(\Rr^d)\cap H^{1,s}(\Rr^d).
\end{equation}
\end{nota}
 It was proved in \cite{LeoTice, Stri} that  there exist  unique continuous trace operators
\bq\label{Tr}
\mathrm{Tr}_{\Omega^\pm\to \Sigma}: \dot H^1(\Omega^\pm)\to \wt H^\mez_\pm(\Sigma)\equiv \wt H^\mez_\pm(\Rr^d)
\eq
with norm depending only on $\| \eta\|_{\dot W^{1, \infty}(\Rr^d)}$ and $\| b^\pm\|_{\dot W^{1, \infty}(\Rr^d)}$.
The Sobolev spaces $\wt H^s_\pm$ are homogeneous and tailored to the boundaries $\Gamma^\pm$. This is crucial for  the two-phase Muskat problem since the traces $f^\pm$ obtained by solving \eqref{system:fpm} are only determined up to additive constants. Employing the lifting results in \cite{LeoTice, Stri} for homogeneous Sobolev spaces, it was proved in \cite{NgPa} that for each $f\in \wt H^\mez_-$, there exists a unique variational solution $\phi$ to \eqref{eq:elliptic}. This in turn implies that $G^-(\eta)f\in H^{-\mez}$ provided that $\eta\in \dot W^{1, \infty}$. 
For continuity estimates in higher Sobolev norms, we shall appeal  to the following theorem.
\begin{theo}[\protect{\cite{ABZ3, NgPa}}]\label{theo:estDN}
Let~$d\ge 1$, $s>1+\frac{d}{2}$ and $\frac 1 2 \le \sigma \leq s$. Consider  $f\in \wt H^\sigma_{-}(\Rr^d)$  and  $\eta\in H^s(\Rr^d)$ with $\dist(\eta, {\Gamma^{-}})\ge h>0$. Then we have 
$G^{-}(\eta)f\in H^{\sigma-1}(\Rr^d)$ and
\begin{equation}\label{est:DN}
\| G^{-}(\eta)f \|_{H^{\sigma-1}}\le 
\mathcal{F}\big(\| \eta \|_{H^s}\big)\Vert f\Vert_{\widetilde{H}^\sigma_-}
\end{equation}
for some $\cF:\Rr^+\to \Rr^+$  depending only on $(s, \sigma, h)$ and $\| \underline b^-\|_{\dot W^{1, \infty}(\Rr^d)}$.
\end{theo}
Since the bottoms $\underline b^\pm$ are fixed in $\dot W^{1, \infty}(\Rr^d)$, we shall omit the dependence on $\| \underline b^\pm\|_{\dot W^{1, \infty}(\Rr^d)}$ in the remainder of this paper.

It is well known that for smooth domains, the Dirichlet-Neumann operator is a first-order pseudo-differential operator whose principal symbol is given by 
\bq\label{ld}
\lambda(x, \xi)=\Big((1+|\nabla\eta(x)|^2)|\xi|^2-(\nabla\eta(x)\cdot \xi)^2\Big)^\mez.
\eq
The one-dimensional case is special since $\lambda(x, \xi)=|\xi|$ is $x$-independent. The following result provides error estimates when paralinearizing $G^-(\eta)$ by $T_\ld$,  which will be the key tool for paralinearizing the Muskat problem with surface tension.
\begin{theo}[\protect{\cite{ABZ3, NgPa}}]\label{paralin:ABZ}
Let $d\ge 1$, $s>1+\frac d2$ and  $\sigma\in [\mez, s-\mez]$. Fix a real number $\delta \in \big(0,  s-1-\frac d2\big)$, $\delta\le \mez$. 
If $f\in \wt H^\sigma_-(\Rr^d)$ and $\eta\in H^s(\Rr^d)$   with $\dist(\eta, \Gamma^-)>h>0$ then  
\begin{align}
&G^{-}(\eta)f=T_\lambda f+R^-(\eta)f,\\
\label{RDN:ABZ}
&\| R^-(\eta)f\|_{H^{\sigma-1+\delta}}\le \cF(\| \eta\|_{H^s})\Vert f\Vert_{\widetilde{H}^\sigma_-}
\end{align}
for some $\cF:\Rr^+\to \Rr^+$  depending only on $(s, \sigma, \delta, h)$.
\end{theo} 
Theorems \ref{theo:estDN} and \ref{paralin:ABZ} were first obtained in \cite{ABZ3} (see Theorem 3.12 and Proposition 3.13 therein) when $f\in H^\sigma$, and extended to $f\in \wt H^\sigma_-$ as a special case of Theorem 3.18 in \cite{NgPa}. It surprisingly  turns out that the case with surface tension  requires a less precise paralinearization  compared to the one needed in \cite{NgPa} for the case without surface tension. This is in contrast with the water waves problem \cite{ABZ1, ABZ3}.

Finally, we will need  contraction estimates for the Dirichlet-Neumann operator  in order to obtain uniqueness and stability of solutions.
\begin{theo}[\protect{\cite[Proposition 3.31]{NgPa}}]\label{contract:DN}
Let  $s>1+\frac{d}{2}$ with $d\ge 1$. Consider $f\in \wt H^{s-\mez}_-(\Rr^d)$ and  $\eta_1$, $\eta_2\in H^s(\Rr^d)$ with $\dist(\eta_j, \Gamma^-)\ge h$ for $j=1, 2$. Then  there exists $\cF:\Rr^+\to \Rr^+$ depending only on $(s, h)$ such that
\begin{equation}\label{contraction:DN}
\| G^{-}(\eta_1)f-G^{-}(\eta_2)f  \|_{H^{s-\tdm}}\le 
\mathcal{F}\big(\| (\eta_1, \eta_2) \|_{H^s}\big)\| \eta_1-\eta_2\|_{H^s}\| f\|_{\wt H^{s-\mez}_{-}}.
\end{equation}
\end{theo}
\section{Proof of Theorem \ref{theo:1p}}\label{sec:1p}
\subsection{Paradifferential reduction}
We assume that $\eta\in Z^s(T)$ with $s>1+\frac d2$  is a solution of \eqref{eq:eta} and satisfies 
\bq\label{cd:h}
\inf_{t\in [0, T]}\dist(\eta(t), \Gamma^-)>h>0.
\eq
The next proposition shows that equation \eqref{eq:eta} can be reduced to an explicit third-order parabolic equation with a smoothing right-hand side.
\begin{prop}\label{prop:reduce}
Set 
\bq\label{def:ell}
\ell(x, \xi)=(1+|\na\eta|^2)^{-\mez}\Big(|\xi|^2-\frac{(\na \eta\cdot \xi)^2}{(1+|\na \eta|^2)}\Big).
\eq
For $\delta\in (0, s-1-\frac{d}{2})$ and $\delta\le \mez$,  there exists $\cF:\Rr^+\to \Rr^+$ depending only on $(h, s, \delta)$ such that
\begin{align}\label{eq:reduce}
&\p_t\eta=-\frac{\s}{\mu^-}T_{\ld \ell}\eta +g,\\ \label{estg:reduce}
&\| g\|_{H^{s-\tdm+\delta}}\le  \cF(\| \eta\|_{H^s})(\frac{\s}{\mu^-}\| \eta\|_{H^{s+\tdm}}+\frac{\rho^-\g}{\mu^-}\|\eta \|_{H^{s-\mez+\delta}}).
\end{align}
\end{prop}
\begin{proof}
Let us rewrite \eqref{eq:eta} as
\bq\label{eq:eta:100}
\p_t\eta=- \frac{\s}{\mu^-} G^{-}(\eta)H(\eta)-\frac{\rho^-\g}{\mu^-} G^{-}(\eta)\eta.
\eq
Theorem \ref{theo:estDN} applied with $\sigma=s-\mez+\delta$ gives 
\bq\label{reduce:estlow}
\| G^-(\eta)\eta\|_{H^{s-\tdm+\delta}}\le \cF(\| \eta\|_{H^s})\Vert \eta\Vert_{H^{s-\mez+\delta}}.
\eq
Regarding $G^-(\eta)H(\eta)$, we apply Theorem \ref{paralin:ABZ} with $\sigma=s-\mez$ and   \eqref{est:F(u):S}  with $s:=s+\mez$ to have
\[
G^{-}(\eta)H(\eta)=T_\lambda H(\eta)+R^-(\eta)H(\eta)
\]
with 
\[
\| R^-(\eta)H(\eta)\|_{H^{s-\tdm+\delta}}\le \cF(\| \eta\|_{H^s})\Vert H(\eta)\Vert_{H^{s-\mez}}\les \cF(\| \eta\|_{H^s})\Vert \eta\Vert_{H^{s+\tdm}}.
\]
The rest of the proof is devoted to control the main term $T_\lambda H(\eta)$. We paralinearize the mean-curvature operator $H(\eta)$ by means of Theorem \ref{paralin:nonl} with $\mu=s+\mez$, $\tau=\delta$:
\bq\label{matrixM}
\frac{\na\eta}{\sqrt{1+|\na\eta|^2}}=T_M\na \eta+f_1,\quad M=\frac{1}{(1+|\na\eta|^2)^\mez}\text{Id}-\frac{\na\eta\otimes \na\eta}{(1+|\na\eta|^2)^\tdm}
\eq
where $\text{Id}$ is the identity matrix and $f_1$ satisfies 
\[
\| f_1\|_{H^{s+\mez+\delta}}\le \cF(\| \eta\|_{H^s})\| \na\eta\|_{H^{s+\mez}}\| \na \eta\|_{C^\delta_*}\les \cF(\| \eta\|_{H^s})\| \eta\|_{H^{s+\tdm}}.
\]
Consequently,
\[
H(\eta)=-\cnx\big(\frac{\na\eta}{\sqrt{1+|\na\eta|^2}}\big)=T_{M\xi\cdot\xi}\eta+T_{-i(\cnx M)\cdot \xi}\eta-\cnx f_1,
\]
where we note that $M\xi \cdot \xi=\ell$. To estimate $T_{-(\cnx M)}\cdot \na\eta$ we use \eqref{boundpara} and the fact that 
\bq\label{estM:Hs-1}
\| M-\text{Id}\|_{H^{s-1}}\le \cF(\|\eta\|_{H^s}),
\eq
 yielding
\[
\begin{aligned}
\|  T_{\cnx M}\cdot \na\eta\|_{H^{s-\mez+\delta}}&\les \| \cnx M\|_{H^{s-2}}\| \na \eta\|_{H^{s+\mez}}\les \cF(\| \eta\|_{H^s})\| \eta\|_{H^{s+\tdm}}.
\end{aligned}
\]
We thus obtain
\bq\label{paralin:H}
\| H(\eta)-T_\ell\eta\|_{H^{s-\mez+\delta}}\le \cF(\| \eta\|_{H^s})\| \eta\|_{H^{s+\tdm}}.
\eq
Since $M^1_\delta(\ld)+M^2_\delta(\ell)\le \cF(\| \eta\|_{H^s})$ (see Lemma \ref{lemm:symbol}), Theorem \ref{theo:sc} (ii) yields that $T_\ld T_\ell-T_{\ld \ell}$ is of order $3-\delta$ and that
\[
\| (T_\ld T_\ell-T_{\ld \ell})\eta\|_{H^{s-\tdm+\delta}}\le\cF(\| \eta\|_{H^s})\| \eta\|_{H^{s+\tdm}}.
\]
Putting together the above considerations we arrive at
\bq
\| G^{-}(\eta)H(\eta)-T_{\ld \ell}\eta\|_{H^{s-\tdm+\delta}}\le\cF(\| \eta\|_{H^s})\| \eta\|_{H^{s+\tdm}}
\eq
which combined with \eqref{reduce:estlow} and \eqref{eq:eta:100} concludes the proof.
\end{proof}
\begin{rema}
In view of \eqref{ld} and \eqref{def:ell} we have 
\bq\label{ldell}
\ld \ell=(1+|\na \eta|^2)^{-\tdm}\ld^3\ge \frac{|\xi|^3}{(1+\|\na \eta\|_{L^\infty})^\tdm}
\eq
which shows that $\ld \ell$ is elliptic so long as $\eta\in \dot W^{1, \infty}$.
\end{rema}
\subsection{ A priori estimates}
Using the reduction in Proposition \ref{prop:reduce} and the symbolic calculus for paradifferential operators, we derive a closed a priori  estimate for $\eta$ in $Z^s(T)$:
\begin{prop}\label{prop:aprioriHs}
Let $s>1+\frac d2$. Assume that $\eta\in Z^s(T)$ is a solution of \eqref{eq:eta} such that \eqref{cd:h} is satisfied. There exists $\cF:\Rr^+\to\Rr^+$ depending only on $(h, s,  \frac{\rho^-\g}{\mu^-}, \frac{\s}{\mu^-})$ such that 
\bq\label{apriori:Hs}
\| \eta\|_{Z^s(T)}\le \cF\Big(\| \eta(0)\|_{H^s}+T\cF\big(\| \eta\|_{Z^s(T)}\big)\Big).
\eq
\end{prop}
\begin{proof}
Denote $\langle D_x\rangle=(\text{Id}-\Delta_x)^\mez$ and $\eta_s=\langle D_x\rangle^s\eta$. Commuting equation \eqref{eq:reduce} with $\langle D_x\rangle^s$ we obtain
\[
\p_t\eta_s=-\frac{\s}{\mu^-}T_{\ld \ell}\eta_s -\frac{\s}{\mu^-}[\langle D_x\rangle^s, T_{\ld \ell}]\eta+\langle D_x\rangle^sg
\]
which yields 
\bq\label{ee:dt}
\mez\frac{d}{dt}\| \eta_s\|_{L^2}^2=-\frac{\s}{\mu^-}(T_{\ld \ell}\eta_s, \eta_s)_{L^2\times L^2}-\frac{\s}{\mu^-}([\langle D_x\rangle^s, T_{\ld \ell}]\eta, \eta_s)_{L^2\times L^2} +(\langle D_x\rangle^sg, \eta_s)_{L^2\times L^2}.
\eq
In view of \eqref{estg:reduce},
\bq\label{ee:t3}
\begin{aligned}
\la (\langle D_x\rangle^sg, \eta_s)_{L^2\times L^2}\ra&\le \| \langle D_x\rangle^sg\|_{H^{-\tdm+\delta}} \|\eta_s\|_{H^{\tdm-\delta}}\\
&\le   (\frac{\s}{\mu^-}+\frac{\rho^-\g}{\mu^-})\cF(\| \eta\|_{H^s})\| \eta\|_{H^{s+\tdm}}\| \eta\|_{H^{s+\tdm-\delta}}.
\end{aligned}
\eq
In light of Theorem \ref{theo:sc} (ii) and  Lemma \ref{lemm:symbol},  $[\langle D_x\rangle^s, T_{\ld \ell}]$ is of order $s+3-\delta$ and that
\[
\|[\langle D_x\rangle^s, T_{\ld \ell}]\eta\|_{H^{-\tdm+\delta}}\le  \cF(\| \eta\|_{H^s})\| \eta\|_{H^{s+\tdm}},
\]
whence
\bq\label{ee:t2}
\la ([\langle D_x\rangle^s, T_{\ld \ell}]\eta, \eta_s)_{L^2\times L^2} \ra\le\cF(\| \eta\|_{H^s})\| \eta\|_{H^{s+\tdm}}\| \eta\|_{H^{s+\tdm-\delta}}.
\eq
Next we  write 
\bq\label{ee:d}
\begin{aligned}
(T_{\ld \ell}\eta_s, \eta_s)_{L^2\times L^2}&=(T_{\sqrt{\ld \ell}}\eta_s, T_{\sqrt{\ld \ell}}\eta_s)_{L^2\times L^2}+\big(T_{\sqrt{\ld \ell}}\eta_s, \big((T_{\sqrt{\ld \ell}})^*-T_{\sqrt{\ld \ell}}\big)\eta_s\big)_{L^2\times L^2}\\
&\quad+\big(\big(T_{\ld \ell}-T_{\sqrt{\ld \ell}}T_{\sqrt{\ld \ell}}\big)\eta_s, \eta_s\big)_{L^2\times L^2}\\
&=I+II+III.
\end{aligned}
\eq
 Applying Theorem \ref{theo:sc} (ii), (iii) and Lemma \ref{lemm:symbol}, we find that $T_{\ld \ell}-T_{\sqrt{\ld \ell}}T_{\sqrt{\ld \ell}}$ and $(T_{\sqrt{\ld \ell}})^*-T_{\sqrt{\ld \ell}}$ are respectively of order $3-\delta$ and $\tdm-\delta$  and that
\[
\| \big(T_{\ld \ell}-T_{\sqrt{\ld \ell}}T_{\sqrt{\ld \ell}}\big)\eta_s\|_{H^{-\tdm+\delta}}+\|\big((T_{\sqrt{\ld \ell}})^*-T_{\sqrt{\ld \ell}}\big)\eta_s\|_{H^{\delta}}\le  \cF(\| \eta\|_{H^s})\| \eta\|_{H^{s+\tdm}}.
\]
Consequently,
\bq\label{ee:d1}
\la II\ra+\la III\ra \le\cF(\| \eta\|_{H^s})\| \eta\|_{H^{s+\tdm}}\| \eta\|_{H^{s+\tdm-\delta}}.
\eq
As for $I$ we first note that the lower bound \eqref{ldell} implies $M_\delta^{-\tdm}(\sqrt{\ld \ell}^{-1})\le \cF(\| \eta\|_{H^s})$ (see Lemma \ref{lemm:symbol}). Theorem \ref{theo:sc} (i) and (ii) then gives that $T_{\sqrt{\ld\ell}^{-1}}T_{\sqrt{\ld\ell}}-T_1=T_{\sqrt{\ld\ell}^{-1}}T_{\sqrt{\ld\ell}}-\Psi(D)$ is of order $-\delta$ ($\Psi$ given by \eqref{cond.psi}) and that 
\[
\begin{aligned}
\| \Psi(D)\eta_s\|_{H^{\tdm}}&\le\|T_{\sqrt{\ld\ell}^{-1}}T_{\sqrt{\ld\ell}}\eta_s\|_{H^\tdm}+\| (\text{Id}-T_{\sqrt{\ld\ell}^{-1}}T_{\sqrt{\ld\ell}})\eta_s\|_{H^{\tdm}}\\
&\le \cF(\| \eta\|_{H^s})\big(\|T_{\sqrt{\ld\ell}}\eta_s\|_{L^2}+\| \eta_s\|_{H^{\tdm-\delta}}\big).
\end{aligned}
\]
It follows that
\[
\| \eta_s\|_{H^{\tdm}}\le \cF(\| \eta\|_{H^s})\big(\|T_{\sqrt{\ld\ell}}\eta_s\|_{L^2}+\| \eta_s\|_{H^{\tdm-\delta}}\big)
\]
and hence,
\bq\label{ee:d2}
I=\|T_{\sqrt{\ld\ell}}\eta_s\|_{L^2}^2\ge \frac{1}{\cF(\| \eta\|_{H^s})}\|\eta \|^2_{H^{s+\tdm}}-\cF(\| \eta\|_{H^s})\| \eta\|_{H^{s+\tdm}}\| \eta\|_{H^{s+\tdm-\delta}}.
\eq
Combining \eqref{ee:d}, \eqref{ee:d1} and \eqref{ee:d2} leads to 
\bq\label{ee:d10}
-(T_{\ld \ell}\eta_s, \eta_s)_{L^2\times L^2}\le -\frac{1}{\cF(\| \eta\|_{H^s})}\|\eta_s \|^2_{H^\tdm}+\cF(\| \eta\|_{H^s})\| \eta\|_{H^{s+\tdm}}\| \eta\|_{H^{s+\tdm-\delta}}
\eq
for some $\cF$ depending only on $(h, s)$. From this, \eqref{ee:dt}, \eqref{ee:t3} and  \eqref{ee:t2} we arrive at
\[
\mez\frac{d}{dt}\| \eta\|_{H^s}^2\le -\frac{\s}{\mu^-}\frac{1}{\cF(\| \eta\|_{H^s})}\|\eta \|^2_{H^{s+\tdm}}+(\frac{\s}{\mu^-}+\frac{\rho^-\g}{\mu^-})\cF(\| \eta\|_{H^s})\| \eta\|_{H^{s+\tdm}}\| \eta\|_{H^{s+\tdm-\delta}}
\]
where $\cF$ depends only on $(h, s)$. The gain of $\delta$ derivative in the second term allows one to interpolate 
\[
\| \eta\|_{H^{s+\tdm}}\| \eta\|_{H^{s+\tdm-\delta}}\les \| \eta\|_{H^{s}}^{1-\tt}\| \eta\|_{H^{s+\tdm}}^{1+\tt},\quad \tt\in (0, 1),
\]
where $\cF:\Rr^+\to\Rr^+$ depends only on $(h, s,  \frac{\rho^-\g}{\mu^-}, \frac{\s}{\mu^-})$. We then  use Young's inequality to hide $\| \eta\|_{H^{s+\tdm}}^{1+\tt}$, leading to 
\[
\mez\frac{d}{dt}\| \eta\|_{H^s}^2\le -\frac{1}{\cF(\| \eta\|_{H^s})}\|\eta \|^2_{H^{s+\tdm}}+\cF(\| \eta\|_{H^s})\| \eta\|_{H^s}^2.
\]
Finally,  a Gr\"onwall argument finishes the proof.
\end{proof}
As the function $\cF$ in \eqref{apriori:Hs} depends on the distance between the surface and the bottom, we need an  a priori estimate for this quantity.
\begin{lemm}\label{lemm:apriorih}
Under the assumptions of Proposition \ref{prop:aprioriHs}, there exist $\tt\in (0, 1)$ and $\cF:\Rr^+\to\Rr^+$ depending only on $(h, s, \frac{\s}{\mu^-}, \frac{\rho^-\g}{\mu^-})$ such that 
\bq\label{apriori:h}
\inf_{t\in [0, T]}\dist(\eta(t), \Gamma^-)\ge \dist(\eta(0), \Gamma^-)- T^\tt\cF(\| \eta\|_{Z^s(T)}).
\eq
\end{lemm}
\begin{proof}
Using equation \eqref{eq:eta}, Theorem \ref{theo:estDN} and the fact that $s+\tdm>3$, we have
\[
\begin{aligned}
\| \eta(t)-\eta(0)\|_{L^2}&\le \int_0^t \|G^-(\eta)\big(\frac{\s}{\mu^-} H(\eta)+\frac{\rho^-\g}{\mu^-}\eta\big)(r)\|_{L^2}dr\\
&\le \int_0^t \cF(\| \eta(r)\|_{H^s})\|\eta(r)\|_{H^3}dr\\
&\le t^{\frac{1}{2}}\cF(\| \eta\|_{L^\infty([0, t]; H^s)})\|\eta\|_{L^2([0, t]; H^{s+\tdm})}.
\end{aligned}
\]
Fixing $s'\in (1+\frac d2, s)$ and using interpolation yields
\[
\begin{aligned}
\| \eta(t)-\eta(0)\|_{H^{s'}}&\le \| \eta(t)-\eta(0)\|_{L^2}^{\tt}\| \eta(t)-\eta(0)\|_{H^{ s}}^{1-\tt}\le t^{\frac{\tt}{2}}\cF(\| \eta\|_{Z^s(t)})
\end{aligned}
\]
for some $\tt\in (0, 1)$. Then in view of the embedding $H^{s'}\subset L^\infty$, this implies \eqref{apriori:h}.
\end{proof}
\subsection{Contraction estimates}
Our goal in this subsection is to prove the following contraction estimate for  solutions of \eqref{eq:eta}.
\begin{theo}\label{theo:contra}
Let $s>1+\frac d2$. Assume that $\eta_1$ and $\eta_2$ are two solutions of \eqref{eq:eta} in $Z^s(T)$ that satisfy \eqref{cd:h}. There exists $\cF:\Rr^+\to \Rr^+$ depending only on $(h, s, \frac{\rho^-\g}{\mu^-}, \frac{\s}{\mu^-})$ such that 
\bq\label{est:contreta}
\| \eta_1-\eta_2\|_{Z^s(T)}\le \cF\big(\|(\eta_1, \eta_2)\|_{Z^s(T)}\big)\| (\eta_1-\eta_2)\vert_{t=0}\|_{H^s}.
\eq
\end{theo}
We first prove a contraction estimate for the remainder in the paralinearization  $H(\eta)\sim T_\ell \eta$.
\begin{lemm}\label{lemm:contraH}
Set 
\bq\label{def:RH}
R_H(\eta)=H(\eta)-T_\ell \eta
\eq
 where $\ell$ is defined in terms of $\eta$ as in \eqref{def:ell}. For $\delta\in(0, s-1-\frac d2)$ and $\delta\le 1$, there exists $\cF$ depending only on $s$ such that
\[
\| R_H(\eta_1)-R_H(\eta_2)\|_{H^{s-\mez}}\le \cF(\| (\eta_1, \eta_2)\|_{H^s})\big(\| \eta_1-\eta_2\|_{H^{s+\tdm-\delta}}+\| (\eta_1, \eta_2)\|_{H^{s+\tdm}}\| \eta_1-\eta_2\|_{H^s}\big).
\]
\end{lemm}
\begin{proof}
We denote the G\^ateaux derivative $d_u F(u)$ of a function $F$ at $u$ in the direction $\dot u$  by 
\[
d_u F(u)\dot u=\lim_{\eps\to 0}\frac{1}{\eps}(F(u+\eps \dot u)-F(u)).
\]
By virtue of the mean-value theorem for G\^ateaux derivative, it suffices to prove that
\bq\label{dRH}
\| d_\eta R_H(\eta)\dot\eta\|_{H^{s-\mez}}\le \cF(\| \eta\|_{H^s})\big(\| \dot \eta\|_{H^{s+\tdm-\delta}}+\| \eta\|_{H^{s+\tdm}}\| \dot \eta\|_{H^s}\big).
\eq
Setting $f(z)=\frac{z}{\sqrt{1+|z|^2}}$ for $z\in \Rr^d$, we write $H(\eta)=-\cnx f(\na \eta)$. Since $d_\eta f(\na \eta)\dot \eta=M\na \dot \eta$, where $M=M(\na \eta)$ is given by \eqref{matrixM}, it follows that
\[
d_\eta R_H(\eta)\dot\eta= -\cnx(M)\na\dot\eta-M\na\cdot \na\dot\eta-T_\ell \dot \eta-T_{d_\eta \ell \dot \eta}\eta.
\]
 Using Bony's decomposition and the fact that $M\xi\cdot \xi=\ell$, we obtain
\begin{align*}
d_\eta R_H(\eta)\dot\eta&= -T_{\cnx(M)}\cdot \na\dot\eta+T_{M\xi\cdot \xi}\dot\eta-g_0-T_\ell \dot \eta-T_{d_\eta \ell(\eta) \dot \eta}\eta\\
&= -T_{\cnx(M)}\cdot \na\dot\eta-g_0-T_{d_\eta \ell(\eta) \dot \eta}\eta
\end{align*}
where $g_0=g_1+g_2$,
\[
g_1=\cnx(M)\na\dot\eta-T_{\cnx(M)}\na\dot\eta,\quad g_2=M\na\cdot \na\dot\eta-T_M\na\cdot \na\dot\eta.
\]
Since
\[
\| M-\text{Id}\|_{H^{s+\mez}}\le \cF(\| \eta\|_{H^s})\| \na\eta\|_{H^{s+\mez}}\le \cF(\| \eta\|_{H^s})\| \eta\|_{H^{s+\tdm}},
\]
  \eqref{paralin:product} implies
\begin{align*}
&\| g_1\|_{H^{s-\mez}}\les\|  \cnx M\|_{H^{s-\mez}}\| \na \dot\eta\|_{H^{s-1}}  \les \cF(\| \eta\|_{H^s})\| \eta\|_{H^{s+\tdm}}\| \dot \eta\|_{H^s},\\
&\| g_2\|_{H^{s-\mez}}\les (\|  M-\text{Id}\|_{H^{s+\mez}}+1)\| \na^2 \dot\eta\|_{H^{s-2}}\les  \cF(\| \eta\|_{H^s})(1+\| \eta\|_{H^{s+\tdm}})\| \dot \eta\|_{H^s}.
\end{align*}
By means of  \eqref{estM:Hs-1} and  \eqref{boundpara} we get
\begin{align*}
&\|  T_{\cnx M}\cdot \na\dot\eta\|_{H^{s-\mez}}\les \| \cnx M\|_{H^{s-2}}\| \na\dot\eta\|_{H^{s+\mez-\delta}}\les \cF(\| \eta\|_{H^s})\| \dot \eta\|_{H^{s+\tdm-\delta}}
\end{align*}
for $\delta\in(0, s-1-\frac d2)$ and $\delta\le 1$.

 Finally,  for $T_{d_\eta \ell(\eta)\dot \eta}\eta$ we note that $d_\eta\ell(\eta)\dot \eta=F(\na \eta, \xi)\na \dot \eta$ where $F$ is  homogeneous of order $2$ in $\xi$. Hence,
\[
M^2_0(d_\eta \ell(\eta)\dot \eta)\le \cF(\| \eta\|_{H^s})\| \dot \eta\|_{H^s}
\]
and thus applying Theorem \ref{theo:sc} (i) gives
\[
\| T_{d_\eta \ell(\eta) \dot \eta}\eta\|_{H^{s-\mez}}\le \cF(\| \eta\|_{H^s})\| \dot \eta\|_{H^s}\| \eta\|_{H^{s+\tdm}}.
\]
Putting together the above estimates we arrive at \eqref{dRH} which completes the proof.
\end{proof}
{\it Proof of Theorem \ref{theo:contra}}

Setting $\eta_\delta=\eta_1-\eta_2$ we have
\begin{align}\label{eq:etad}
&\p_t\eta_\delta=-\frac{\s}{\mu^-}G^-(\eta_1)(H(\eta_1)-H(\eta_2))-\cR_0,\\
&\cR_0:=\frac{\rho^-\g}{\mu^-}G^-(\eta_1)\eta_\delta+[G^-(\eta_1)-G^-(\eta_2)]\big(\frac{\s}{\mu^-}H(\eta_2)+\frac{\rho^-\g}{\mu^-} \eta_2\big).
\end{align}
According to Theorem \ref{theo:estDN},
\[
\| G^-(\eta_1)\eta_\delta\|_{H^{s-\tdm}}\le \cF(\|\eta_1\|_{H^s})\Vert \eta_\delta\Vert_{H^{s-\mez}}.
\]
On the other hand, Theorem \ref{contract:DN} applied with $f=\frac{\s}{\mu^-}H(\eta_2)+\frac{\rho^-\g}{\mu^-} \eta_2\in H^{s-\mez}$ gives
\[
 \| [G^-(\eta_1)-G^-(\eta_2)]\big(\frac{\s}{\mu^-}H(\eta_2)+\frac{\rho^-\g}{\mu^-} \eta_2\big)\|_{H^{s-\tdm}}\le (\frac{\s}{\mu^-}+\frac{\rho^-\g}{\mu^-})\cF(N_s)\Vert \eta_\delta\Vert_{H^s}\| \eta_2\|_{H^{s+\tdm}},
\]
where $\cF$ depends only on  $(h, s)$ and we denoted
\bq\label{Nr}
N_r=\|(\eta_1, \eta_2)\|_{H^r}.
\eq
Consequently,
\bq\label{est:cR0}
\| \cR_0\|_{H^{s-\tdm}}\le (\frac{\s}{\mu^-}+\frac{\rho^-\g}{\mu^-})\cF(N_s)\| \eta_\delta\|_{H^s}\big(1+N_{s+\tdm}\big).
\eq
Next we claim that for some $\cF$ depending only on  $(h, s)$,
\begin{align}\label{contra:lin}
&G^-(\eta_1)(H(\eta_1)-H(\eta_2))=T_{\ld_1\ell_1}\eta_\delta+\cR_1,\\ \label{est:cR1}
&\| \cR_1\|_{H^{s-\tdm}}\le \cF(N_s)\big(\| \eta_\delta\|_{H^{s+\tdm-\delta}}+N_{s+\tdm}\| \eta_\delta\|_{H^s}\big).
\end{align}
To this end, let us fix $\delta\in(0, s-1-\frac d2)$ and $\delta\le \mez$. Applying Theorem \ref{paralin:ABZ} with $\sigma=s-\mez-\delta$ we obtain
\begin{align*}
&G^-(\eta_1)(H(\eta_1)-H(\eta_2))=T_{\ld_1}(H(\eta_1)-H(\eta_2))+\cR_2,\\
&\| \cR_2\|_{H^{s-\tdm}}\le \cF(\| \eta_1\|_{H^s})\Vert H(\eta_1)-H(\eta_2)\Vert_{H^{s-\mez-\delta}}.
\end{align*}
In addition, Theorem \ref{est:diffF(U)} together with the embedding $\na\eta_j\in H^{s+\mez-\delta}\subset L^\infty$ implies 
\[
\Vert H(\eta_1)-H(\eta_2)\Vert_{H^{s-\mez-\delta}}\le \Big\| \frac{\na\eta_1}{\sqrt{1+|\na\eta_1|^2}}-\frac{\na\eta_2}{\sqrt{1+|\na\eta_2|^2}}\Big\|_{H^{s+\mez-\delta}}\le \cF(N_s)\Vert \eta_\delta\Vert_{H^{s+\tdm-\delta}},
\]
whence 
\[
\| \cR_2\|_{H^{s-\tdm}}\le\cF(N_s)\Vert \eta_\delta\Vert_{H^{s+\tdm-\delta}}.
\]
Next we write 
\[
\begin{aligned}
T_{\ld_1}(H(\eta_1)-H(\eta_2))&=T_{\ld_1}T_{\ell_1}\eta_\delta+T_{\ld_1}T_{\ell_1-\ell_2}\eta_2+T_{\ld_1}(R_H(\eta_1)-R_H(\eta_2))\\
&=T_{\ld_1\ell_1}\eta_\delta+(T_{\ld_1}T_{\ell_1}-T_{\ld_1\ell_1})\eta_\delta+T_{\ld_1}T_{\ell_1-\ell_2}\eta_2+T_{\ld_1}(R_H(\eta_1)-R_H(\eta_2)).
\end{aligned}
\]
By  Theorem \ref{theo:sc} (i) and Lemma \ref{lemm:contraH}, 
\begin{align*}
&\| T_{\ld_1}(R_H(\eta_1)-R_H(\eta_2))\|_{H^{s-\tdm}}\le \cF(N_s)\big(\| \eta_\delta\|_{H^{s+\tdm-\delta}}+N_{s+\tdm}\| \eta_\delta\|_{H^s}\big).
\end{align*}
Since
\[
M^2_0(\ell_1-\ell_2)\le \cF(N_s)\| \eta_\delta\|_{H^s}
\]
(see Lemma \ref{lemm:symbol}), Theorem \ref{theo:sc} (i) gives 
\[
\| T_{\ld_1}T_{\ell_1-\ell_2}\eta_2\|_{H^{s-\tdm}}\le \cF(N_s)\| \eta_\delta\|_{H^s}\| \eta_2\|_{H^{s+\tdm}}.
\]
Finally, Theorem \ref{theo:sc} (ii) yields that $T_{\ld_1}T_{\ell_1}-T_{\ld_1\ell_1}$ is of order $3-\delta$ and 
\bq\label{Tld1l1}
\| (T_{\ld_1}T_{\ell_1}-T_{\ld_1\ell_1})\eta_\delta\|_{H^{s-\tdm}}\le  \cF(\| \eta_1\|_{H^s})\| \eta_\delta\|_{H^{s+\tdm-\delta}}.
\eq
The above estimates together imply
\bq\label{TldHd}
\begin{aligned}
&T_{\ld_1}(H(\eta_1)-H(\eta_2))=T_{\ld_1\ell_1}\eta_\delta+\cR_3,\\
&\|\cR_3\|_{H^{s-\tdm}}\le\cF(N_s)\big(\| \eta_\delta\|_{H^{s+\tdm-\delta}}+N_{s+\tdm}\| \eta_\delta\|_{H^s}\big).
\end{aligned}
\eq
Therefore, we arrive at \eqref{contra:lin}-\eqref{est:cR1} with $\cR_1=\cR_2+\cR_3$.

Now it follows from equations \eqref{eq:etad}, \eqref{contra:lin} and the estimates  \eqref{est:cR0}, \eqref{est:cR1} that
\bq\label{req:etad}
\p_t\eta_\delta=-\frac{\s}{\mu^-}T_{\ld_1\ell_1}\eta_\delta+\wt \cR_1,
\eq
where $\wt \cR_1=-\frac{\s}{\mu^-}\cR_1-\cR_0$ satisfies 
\bq\label{est:wtcR1}
\| \wt\cR_1\|_{H^{s-\tdm}}\le (\frac{\s}{\mu^-}+\frac{\rho^-\g}{\mu^-})\cF(N_s)\big(\| \eta_\delta\|_{H^{s+\tdm-\delta}}+N_{s+\tdm}\| \eta_\delta\|_{H^s}\big),
\eq
where  $\cF$ depends only on $(h, s)$. An $H^s$ energy estimate for \eqref{req:etad} yields
\bq\label{dt:eetad}
\mez\frac{d}{dt}\| \eta_\delta\|_{H^s}^2\le -\frac{\s}{\mu^-}(T_{\ld_1\ell_1}\eta_\delta, \eta_\delta)_{H^s, H^s}+\|\wt \cR_1\|_{H^{s-\tdm}}\| \eta_\delta\|_{H^{s+\tdm}}.
\eq
The argument leading to \eqref{ee:d10} gives
\bq\label{ee:etad}
\begin{aligned}
 -(T_{\ld_1\ell_1}\eta_\delta, \eta_\delta)_{H^s, H^s}&\le -\frac{1}{\cF(\| \eta_1\|_{H^s})}\|\eta_\delta \|^2_{H^{s+\tdm}}+\cF(\| \eta_1\|_{H^s})\| \eta_\delta\|_{H^{s+\tdm}}\| \eta_\delta\|_{H^{s+\tdm-\delta}}.
\end{aligned}
\eq
Combining \eqref{dt:eetad}, \eqref{ee:etad} and \eqref{est:wtcR1}  we obtain 
\bq
\begin{aligned}
\mez\frac{d}{dt}\| \eta_\delta\|_{H^s}^2&\le -\frac{1}{\cF(N_s)}\|\eta_\delta \|^2_{H^{s+\tdm}}+\cF(N_s)\| \eta_\delta\|_{H^{s+\tdm}}\| \eta_\delta\|_{H^{s+\tdm-\delta}}+\cF(N_s)N_{s+\tdm}\| \eta_\delta\|_{H^{s}}\| \eta_\delta\|_{H^{s+\tdm}}
\end{aligned}
\eq
for some function $\cF$ depending only on $(h, s, \frac{\rho^-\g}{\mu^-}, \frac{\s}{\mu^-})$. By interpolation and Young's inequality we have
\begin{align*}
&\cF(N_s)\| \eta_\delta\|_{H^{s+\tdm}}\| \eta_\delta\|_{H^{s+\tdm-\delta}}\le \frac{1}{10\cF(N_s)}\|\eta_\delta \|^2_{H^{s+\tdm}}+\cF_1(N_s)\|\eta_\delta \|^2_{H^s},\\
&\cF(N_s)N_{s+\tdm}\| \eta_\delta\|_{H^{s}}\| \eta_\delta\|_{H^{s+\tdm}}\le \frac{1}{10\cF(N_s)}\|\eta_\delta \|^2_{H^{s+\tdm}}+\cF_2(N_s)N_{s+\tdm}^2\|\eta_\delta \|^2_{H^s}.
\end{align*}
It follows that
\bq
\frac{d}{dt}\| \eta_\delta\|_{H^s}^2\le -\frac{1}{\cF(N_s)}\|\eta_\delta \|^2_{H^{s+\tdm}}+\cF(N_s)N_{s+\tdm}^2\|\eta_\delta \|^2_{H^s}
\eq
for some $\cF$ depending only on $(h, s, \frac{\rho^-\g}{\mu^-}, \frac{\s}{\mu^-})$. Finally, since 
\[
\int_0^T N_{s+\tdm}^2(t)dt\le \|(\eta_1, \eta_2)\|_{Z^s(T)}^2,
\]
 a simple Gr\"onwall argument leads to \eqref{est:contreta}. 
\subsection{Proof of Theorem \ref{theo:1p}}
Consider an initial datum $\eta_0\in H^s(\Rr^d)$, $s>1+\frac d2$, satisfying $\dist(\eta_0, \Gamma^-)>2h>0$. We construct the sequence of approximate solutions $\eta_\eps$, $\eps\in (0, 1)$, that solve the ODE
\bq
\p_t\eta_\eps=-\frac{1}{\mu^-}J_\eps \Big[G^-(J_\eps \eta_\eps)\big(\s H(J_\eps \eta_\eps)+\rho^-\g J_\eps \eta_\eps\big)\Big],\quad \eta_\eps\vert_{t=0}=\eta_0,
\eq
where $J_\eps$ denotes the usual mollifier that cut off frequencies of size greater than $\eps^{-1}$. Each $\eta_\eps$ exists on some maximal time interval $[0, T_\eps)$ in light of the Cauchy-Lipschitz theorem and Theorems \ref{theo:estDN} and \ref{contract:DN} for the Dirichlet-Neumann operator. It is easy to check that the a priori estimates in Proposition \ref{prop:aprioriHs} and Lemma \ref{lemm:apriorih} remain valid for $\eta_\eps$. Consequently, a continuity argument guarantees the existence of a positive time $T$ such that $T<T_\eps$ for all $\eps\in (0, 1)$ and that on $[0, T]$ the uniform estimates
\begin{align}
&\| \eta_\eps\|_{Z^s(T)}\le \cF(\| \eta_0\|_{H^s}),\quad\inf_{t\in [0, T]}\dist(\eta_\eps(t), \Gamma^-)> h
\end{align}
hold for some $\cF$ depending only on $(h, s, \frac{\rho^-\g}{\mu^-}, \frac{\s}{\mu^-})$.  Theorem \ref{theo:contra} also holds for  $\eta_\eps$, giving  that the sequence $(\eta_\eps)$ is Cauchy in $Z^s(T)$ and thus converges to some $\eta\in Z^s(T)$. By virtue of  Theorems \ref{est:DN} and \ref{contract:DN} we can pass to the limit $\eps \to 0$ and obtain that $\eta$ is a solution of \eqref{eq:eta} with initial data $\eta_0$. Finally, uniqueness and stability follow at once from Theorem \ref{theo:contra}.  
\section{Proof of Theorem \ref{theo:2p}}\label{sec:2p}
\subsection{Regularity of $f^\pm$}
We first recall the well-posedness of variational solutions to \eqref{system:fpm}.
\begin{prop}[\protect{\cite[Proposition 4.8 and Remark 4.9]{NgPa}}]\label{prop:fpm}
Let $\eta \in W^{1, \infty}(\Rr^d)\cap H^\mez(\Rr^d)$ satisfy $\dist(\eta, \Gamma^\pm)>h>0$. Then there exists a unique variational solution $f^\pm \in \wt H^\mez_\pm(\Rr^d)$ to the system \eqref{system:fpm}. Moreover, $f^\pm$ satisfy 
\bq\label{variest:fpm:0}
\| f^\pm\|_{\wt H^\mez_\pm}\le C(1+\| \eta\|_{W^{1, \infty}})^2\| \s H(\eta)+\lb \rho\rb\g \eta\|_{H^\mez}
\eq
where the constant $C$ depends only on $(h, \mu^\pm)$.
\end{prop}
It follows from \eqref{variest:fpm:0} and Theorem \ref{est:nonl} that
\bq\label{variest:fpm}
\| f^\pm\|_{\wt H^\mez_\pm}\le \cF(\| \eta\|_{W^{1, \infty}})(\s\| \eta\|_{H^{\frac 52}}+\lb \rho\rb \g \| \eta\|_{H^\mez})
\eq
for some function $\cF$ depending only on $(h, \mu^\pm)$. Using the variational estimate \eqref{variest:fpm} and the paralinearization Theorem \ref{paralin:ABZ}, we prove that higher Sobolev regularity for $f^\pm$ can be transferred from $\eta$.
\begin{prop}\label{prop:estfpm}
Let $f^\pm$ be the solution of \eqref{system:fpm} as given by Proposition \ref{prop:fpm}. If $\eta\in H^{s+\tdm}(\Rr^d)$ with $s>1+\frac d2$ then $f^\pm\in \wt H^{s-\mez}_\pm(\Rr^d)$ and 
\bq\label{est:fpm}
\| f^\pm\|_{\wt H^r_\pm}\le \cF(\| \eta\|_{H^s})(\s\| \eta\|_{H^{r+2}}+\lb\rho\rb\g\| \eta\|_{H^r})
\eq
for all $r\in [\mez, s-\mez]$, where $\cF$ depends only on $(h, s, r, \mu^\pm)$.
\end{prop}
\begin{proof}
Fix $\delta \in (0, s-1-\frac d2)$ and $\delta\le \mez$. First, we claim that for $\sigma\in [\mez, s-\mez-\delta]$, if $f^\pm\in \wt H^\sigma_\pm$ then there exists $\cF$ depending only on $(h, s, \sigma, \delta, \mu^\pm)$ such that
\bq\label{bootstrapf:0}
\| T_\ld f^\pm\|_{H^{\sigma-1+\delta}}\le \cF(\| \eta\|_{H^s})\| f^\pm\|_{\wt H^\sigma_\pm}+\cF(\| \eta\|_{H^s})(\s\| \eta\|_{H^{\sigma+2+\delta}}+\lb\rho\rb\g\| \eta\|_{H^{\sigma+\delta}}).
\eq
Indeed, according to Theorem \ref{paralin:ABZ} there exists $\cF$ depending only on $(h, s, \sigma, \delta)$ such that
\begin{align*}
&G^\pm (\eta)f^\pm=\mp T_\lambda f^\pm+R^\pm(\eta)f^\pm,\\
&\| R^\pm(\eta)f^\pm\|_{H^{\sigma-1+\delta}}\le \cF(\| \eta\|_{H^s})\Vert f^\pm\Vert_{\wt H^\sigma_\pm}.
\end{align*}
Then using the  system  \eqref{system:fpm} we obtain after rearranging terms that 
\bq\label{Tld:f-}
T_\ld f^-=\frac{\mu^-}{\mu^++\mu^-}T_\ld (\s H(\eta)+\lb \rho\rb\g \eta)+\frac{\mu^-}{\mu^++\mu^-}R^+(\eta)f^+-\frac{\mu^+}{\mu^++\mu^-}R^-(\eta)f^-
\eq
 which together with Theorem \ref{theo:sc} (i) and the bound
  \[
  \|\s H(\eta)+\lb \rho\rb\g \eta\|_{H^{\sigma+\delta}} \le \cF(\| \eta\|_{H^s})(\s\| \eta\|_{H^{\sigma+2+\delta}}+\lb\rho\rb\g\| \eta\|_{H^{\sigma+\delta}})
  \]
proves the claim \eqref{bootstrapf:0}. Note that $\sigma+2+\delta \in [\frac 52+\delta, s+\tdm]$.
  
 We now bootstrap the regularity for $f^\pm$ using \eqref{bootstrapf:0} and the inequality
 \bq\label{ineq}
\| u\|_{H^{1, \mu}}\les \| u\|_{H^{1, \mez}}+ \mathcal{F}(\Vert \eta\Vert_{H^s})(\| T_\ld u\|_{H^{\mu-1}}+\| u\|_{H^{1, \mu -\delta}}),\quad\mu \ge \mez.
\eq
Let us  first prove \eqref{ineq}.  By virtue of Theorem \ref{theo:sc} (ii) and Remark \ref{rema:low}, we have for $\mu \in \Rr$,
\bq\label{ineq:0}
\begin{aligned}
\| \Psi(D_x)u\|_{H^\mu}=\| T_1u\|_{H^\mu}&\le \| T_{\ld^{-1}}T_\ld u\|_{H^\mu}+\|T_1-T_{\ld^{-1}}T_\ld u\|_{H^\mu}\\
& \le \mathcal{F}(\Vert \eta\Vert_{H^s})(\| T_\ld u\|_{H^{\mu-1}}+\| u\|_{H^{1, \mu -\delta}}),
\end{aligned}
\eq
where the cut-off $\Psi$ removing the low frequency part  is defined by \eqref{cond.psi}. On the other hand, for $\mu \ge \mez$ we have
\[
\| u\|_{H^{1, \mu}}\les \| u\|_{H^{1, \mez}}+\| \Psi(D_x) u\|_{H^\mu}
\]
which combined with \eqref{ineq:0} yields  \eqref{ineq}. Now applying \eqref{ineq} with $\mu=\sigma+\delta$, $\sigma=\mez$ and invoking \eqref{variest:fpm} and \eqref{bootstrapf:0} we deduce that 
\[
\| f^\pm\|_{H^{1, \mez+\delta}}\le \cF(\| \eta\|_{H^s})(\s\| \eta\|_{H^{\frac 52+\delta}}+\lb\rho\rb\g\| \eta\|_{H^{\mez+\delta}})
\]
where $\cF$ depends only on $(h, s, \sigma, \delta, \mu^\pm)$. We have thus bootstrapped the regularity of $f^\pm$ from $H^{1, \mez}$ to $H^{1, \mez+\delta}$ by using \eqref{bootstrapf:0} with $\sigma=\mez$. Since \eqref{bootstrapf:0} holds for $\sigma\in [\mez, s-\mez-\delta]$, an induction argument leads to 
\[
\| f^\pm\|_{H^{1, r}}\le \cF(\| \eta\|_{H^s})(\s\| \eta\|_{H^{r+2}}+\lb\rho\rb\g\| \eta\|_{H^r})
\]
for all $r\in [\mez, s-\mez]$. In conjunction  with \eqref{variest:fpm} and the definition \eqref{def:wtHs} of $\wt H^r_\pm$, this yields \eqref{est:fpm}. 
\end{proof}
\begin{rema}
The estimate \eqref{est:fpm} shows that $f^\pm$ behave like $\s H(\eta)+\lb \rho\rb \g\eta$.
\end{rema}
\subsection{Paradifferential reduction and a priori estimates}
Assume that  $\eta\in Z^s(T)$ with $s>1+\frac d2$ solves \eqref{eq:eta2p} and satisfies
\bq\label{cdh:2p}
\inf_{t\in [0, T]}\dist(\eta(t), \Gamma^\pm)>h>0.
\eq
Moreover,  let  $f^\pm\in \wt H^{s-\mez}_\pm$ be the solution of \eqref{system:fpm} as given by Propositions  \ref{prop:fpm} and \ref{prop:estfpm}.
\begin{prop}\label{prop:reduce:2p}
For $\delta \in \big(0,  s-1-\frac d2\big)$, $\delta \le \mez$,  there exists $\cF:\Rr^+\to \Rr^+$ depending only on $(h, s, \delta, \mu^\pm)$ such that
\begin{align}\label{eta:reduce2p}
&\p_t\eta=\frac{-\s }{\mu^++\mu^-}T_{\ld \ell} \eta+g,\\ \label{estg:2p}
&\| g\|_{H^{s-\tdm+\delta}}\le \cF(\| \eta\|_{H^s})(\s\| \eta\|_{H^{s+\tdm}}+\lb \rho\rb\g \|\eta\|_{H^{s-\mez+\delta}}).
\end{align}
\end{prop}
\begin{proof}
We rewrite \eqref{Tld:f-} as
\begin{align*}
T_\ld f^-
&=\frac{\s\mu^-}{\mu^++\mu^-}T_{\ld \ell} \eta+\frac{\s\mu^-}{\mu^++\mu^-}(T_\ld  T_\ell \eta-T_{\ld \ell})\eta +\frac{\s\mu^-}{\mu^++\mu^-}T_\ld (H(\eta)-T_\ell \eta)\\
&\quad+ \frac{\lb \rho\rb\g \mu^-}{\mu^++\mu^-}T_\ld\eta+\frac{\mu^-}{\mu^++\mu^-}R^+(\eta)f^+-\frac{\mu^+}{\mu^++\mu^-}R^-(\eta)f^-,
\end{align*}
where by virtue of Theorem \ref{paralin:ABZ} and Proposition \ref{prop:estfpm},
\bq\label{est:Rfpm}
\begin{aligned}
\| R^\pm(\eta)f^\pm\|_{H^{s-\tdm+\delta}}&\le\cF(\| \eta\|_{H^s})\| f^\pm\|_{\wt H_\pm^{s-\mez}}\\
&\les \cF(\| \eta\|_{H^s})(\s\| \eta\|_{H^{s+\tdm}}+\lb \rho\rb\g \| \eta\|_{H^{s-\mez}}).
\end{aligned}
\eq
Using  \eqref{paralin:H} and Theorem \ref{theo:sc} (i) (ii), we can bound 
\begin{align*}
&\| (T_\ld T_\ell-T_{\ld \ell})\eta\|_{H^{s-\tdm+\delta}}+\| T_{\ld}(H(\eta)-T_\ell\eta)\|_{H^{s-\tdm+\delta}}\le \cF(\| \eta\|_{H^s})\| \eta\|_{H^{s+\tdm}},\\
&\| T_\ld\eta\|_{H^{s-\tdm+\delta}}\le \cF(\| \eta\|_{H^s})\| \eta\|_{H^{s-\mez+\delta}}.
\end{align*}
We thus obtain
 \begin{align*}
&T_\ld f^-=\frac{\s \mu^-}{\mu^++\mu^-}T_{\ld \ell} \eta+g_0,\\
&\| g_0\|_{H^{s-\tdm+\delta}}\le \cF(\| \eta\|_{H^s})(\s\| \eta\|_{H^{s+\tdm}}+\lb \rho\rb\g \|\eta\|_{H^{s-\mez+\delta}}),
\end{align*}
for some $\cF$ depending only on $(h, s, \delta, \mu^\pm)$. Plugging this into the paralinearization  
\[
G^-(\eta)f^-=T_\ld f^-+R^-(\eta)f^-
\]
 and using \eqref{est:Rfpm} and \eqref{eq:eta2p} we conclude the proof.
\end{proof}
It follows from \eqref{estg:2p} that 
\bq
\| g\|_{H^{s-\tdm+\delta}}\le (\s+\lb \rho\rb \g)\cF(\| \eta\|_{H^s})\| \eta\|_{H^{s+\tdm}}.
\eq
We have thus reduced the two-phase Muskat problem to the paradifferential parabolic equation \eqref{eta:reduce2p} which is of the same form as equation \eqref{eq:reduce} for the one-phase problem. Therefore, the proofs of Proposition \ref{prop:aprioriHs} and Lemma \ref{lemm:apriorih} yield the following  a priori estimates.
\begin{prop}\label{apriori:2p}
There exist $\tt\in (0, 1)$ depending only on $s$ and  $\cF:\Rr^+\to \Rr^+$ depending only on $(h, s, \s, \mu^\pm, \lb \rho\rb\g)$ such that
\bq
\| \eta\|_{Z^s(T)}\le \cF\Big(\| \eta(0)\|_{H^s}+T\cF\big(\| \eta\|_{Z^s(T)}\big)\Big)
\eq
and
\bq
\inf_{t\in [0, T]}\dist(\eta(t), \Gamma^-)\ge \dist(\eta(0), \Gamma^-)-T^\tt\cF(\| \eta\|_{Z^s(T)}).
\eq
\end{prop}
\subsection{Contraction estimates}
Considering two solutions $\eta_1$ and $\eta_2$ in $Z^s(T)$ of \eqref{eq:eta2p} that satisfy condition \eqref{cdh:2p}, we prove a contraction estimate in $Z^s(T)$ for the difference $\eta_1-\eta_2$. 
\begin{theo}\label{theo:contra:2p}
There exists $\cF:\Rr^+\to \Rr^+$ depending only on  $(h, s, \s, \mu^\pm, \lb \rho\rb\g)$ such that
\bq\label{contra:eta:2p}
\| \eta_1-\eta_2\|_{Z^s(T)}\le \cF\big(\|(\eta_1, \eta_2)\|_{Z^s(T)}\big)\| (\eta_1-\eta_2)\vert_{t=0}\|_{H^s}.
\eq
\end{theo}
\subsubsection{Contraction estimates for $f^\pm$}
For $j=1, 2$ let $f^\pm_j$ solve
\bq\label{system:fpmj}
\begin{cases}
 f^-_j-f^+_j=k_j:=\s H(\eta_j)+\lb \rho\rb\g\eta_j,\\
\frac{1}{\mu^+}G^+(\eta_j)f^+_j-\frac{1}{\mu^-}G^-(\eta_j)f^-_j=0.
\end{cases}
\eq
We set $f^\pm_\delta=f^\pm_1-f^\pm_2$, $k_\delta=k_1-k_2$, $\eta_\delta=\eta_1-\eta_2$, where the subscript $\delta$ only signifies the difference. We also  recall the notation \eqref{Nr}
\[
N_r=\| (\eta_1, \eta_2)\|_{H^r}.
\]
\begin{lemm}Let $\delta\in (0, s-1-\frac d2)$ and $\delta\le\mez$.

1) For each $r\in [\mez, s-\mez]$, there exists $\cF$ depending only on $(h, s, r, \mu^\pm)$ such that
\bq\label{contra:fpmd}
\begin{aligned}
\| f^\pm_\delta\|_{\wt H^r_\pm}&\le \cF(N_s)(\s\| \eta_\delta\|_{H^{r+2}}+\lb \rho\rb\g\| \eta_\delta\|_{H^r})\\
&\quad +\cF(N_s)\| \eta_\delta\|_{H^s}\big(\s (N_{s+\tdm}+1)+\lb \rho\rb\g N_{s-\mez}\big).
\end{aligned}
\eq
2) For each $\sigma\in [\mez, s-\mez-\delta]$, there exists $\cF$ depending only on $(h, s, \sigma, \mu^\pm)$ such that
\bq\label{Tld:f-eta}
T_{\ld_1}f^-_\delta=\frac{\mu^-}{\mu^++\mu^-}T_{\ld_1}k_\delta +g_-
\eq
with $g_-$ satisfying 
\bq\label{estTld:f-eta}
\begin{aligned}
\| g_-\|_{H^{\sigma-1+\delta}}&\le \cF(N_s)(\s\| \eta_\delta\|_{H^{\sigma+2}}+\lb \rho\rb\g\| \eta_\delta\|_{H^\sigma})\\
&\quad +\cF(N_s)\| \eta_\delta\|_{H^s}\big(\s (N_{s+\tdm}+1)+\lb \rho\rb\g N_{s-\mez}\big).
\end{aligned}
\eq
\end{lemm}
\begin{proof}
Taking the difference of the second equation in \eqref{system:fpmj} for $j=1$ and $j=2$ we find that
\[
\frac{1}{\mu^-}G^-(\eta_1)f^-_\delta-\frac{1}{\mu^+}G^+(\eta_1)f^+_\delta=\frac{1}{\mu^+}[G^+(\eta_1)-G^+(\eta_2)]f^+_2-\frac{1}{\mu^-}[G^-(\eta_1)-G^-(\eta_2)]f^-_2.
\]
Since $G^\pm(\eta_1)f^\pm_\delta=\mp T_{\ld_1} f^\pm_\delta+R^\pm(\eta_1)f^\pm_\delta$ and $f^+_\delta=f^-_\delta-k_\delta$, this gives
\bq\label{Tld:f-eta:0}
T_{\ld_1}f^-_\delta=\frac{\mu^-}{\mu^++\mu^-}T_{\ld_1}k_\delta+\frac{\mu^+\mu^-}{\mu^++\mu^-} F
\eq
where
\[
F=\frac{1}{\mu^+}R^+(\eta_1)f^+_\delta-\frac{1}{\mu^-}R^-(\eta_1)f^-_\delta+\frac{1}{\mu^+}[G^+(\eta_1)-G^+(\eta_2)]f^+_2-\frac{1}{\mu^-}[G^-(\eta_1)-G^-(\eta_2)]f^-_2.
\]
Theorems \ref{theo:sc} (i) and \ref{est:diffF(U)} together imply that for $\nu\in [-1, s-\tdm]$,
\[
\| T_{\ld_1}k_\delta\|_{H^\nu}\le \cF(N_s)\big(\s(\| \eta_\delta\|_{H^{\nu+3}}+\| \eta_\delta\|_{H^s})+\lb \rho\rb\g\| \eta_\delta\|_{H^{\nu+1}}\big).
\]
 In light of Theorem \ref{paralin:ABZ} we have that for $\sigma\in [\mez, s-\mez]$,
 \[
\| R^\pm(\eta_1)f^\pm_\delta\|_{H^{\sigma-1+\delta}}\le \cF(N_s)\| f^\pm_\delta\|_{\wt H^\sigma_\pm}.
\]
Finally, a combination of Theorem \ref{contraction:DN} and Proposition \ref{prop:estfpm} yields
\begin{align*}
\| [G^\pm(\eta_1)-G^\pm(\eta_2)]f^\pm_2\|_{H^{s-\tdm}}&\le \cF(N_s)\| \eta_\delta\|_{H^s}\| f^\pm_2\|_{\wt H^{s-\mez}_\pm}\\
&\les\cF(N_s)\| \eta_\delta\|_{H^s}(\s\| \eta_2\|_{H^{s+\tdm}}+\lb \rho\rb\g \| \eta_2\|_{H^{s-\mez}}).
\end{align*}
Consequently, for $\sigma\in[\mez, s-\mez-\delta]$ we have
\bq\label{estF:fd}
\| F\|_{H^{\sigma-1+\delta}}\le  \cF(N_s)\| f^\pm_\delta\|_{\wt H_\pm^\sigma}+\cF(N_s)\| \eta_\delta\|_{H^s}\big(\s(\| \eta_2\|_{H^{s+\tdm}}+1)+\lb \rho\rb\g \| \eta_2\|_{H^{s-\mez}}\big)
\eq
and 
\begin{align*}
\|T_{\ld_1}f^-_\delta\|_{H^{\sigma-1+\delta}}&\le  \cF(N_s)\| f^\pm_\delta\|_{\wt H_\pm^\sigma}+ \cF(N_s)(\s\| \eta_\delta\|_{H^{\sigma+2+\delta}}+\lb \rho\rb\g\| \eta_\delta\|_{H^{\sigma+\delta}})\\
&\quad +\cF(N_s)\| \eta_\delta\|_{H^s}\big(\s(\| \eta_2\|_{H^{s+\tdm}}+1)+\lb \rho\rb\g \| \eta_2\|_{H^{s-\mez}}\big).
\end{align*}
Invoking the relation $f^+_\delta=f^-_\delta-k_\delta$ leads to the same bound for $\|T_{\ld_1}f^+_\delta\|_{H^{\sigma-1+\delta}}$ and thus
\begin{align*}
\|T_{\ld_1}f^\pm_\delta\|_{H^{\sigma-1+\delta}}&\le  \cF(N_s)\| f^\pm_\delta\|_{\wt H_\pm^\sigma}+ \cF(N_s)(\s\| \eta_\delta\|_{H^{\sigma+2+\delta}}+\lb \rho\rb\g\| \eta_\delta\|_{H^{\sigma+\delta}})\\
&\quad +\cF(N_s)\| \eta_\delta\|_{H^s}\big(\s(\| \eta_2\|_{H^{s+\tdm}}+1)+\lb \rho\rb\g \| \eta_2\|_{H^{s-\mez}}\big)
\end{align*}
 for $\sigma\in[\mez, s-\mez-\delta]$. Now we can apply \eqref{ineq} and use the definition  of $\wt H^\sigma_\pm$ (see \eqref{def:wtHs}) to have 
\bq\label{key:estfd}
\begin{aligned}
\| f^\pm_\delta\|_{H^{1, \sigma+\delta}}&\le \cF(N_s)(\|  f^\pm_\delta\|_{\wt H^\mez_\pm}+\|  f^\pm_\delta\|_{ H^{1, \sigma}})+\cF(N_s)(\s\| \eta_\delta\|_{H^{\sigma+2+\delta}}+\lb \rho\rb\g\| \eta_\delta\|_{H^{\sigma+\delta}})\\
&\quad +\cF(N_s)\| \eta_\delta\|_{H^s}\big(\s(\| \eta_2\|_{H^{s+\tdm}}+1)+\lb \rho\rb\g \| \eta_2\|_{H^{s-\mez}}\big)
\end{aligned}
\eq
for $\sigma\in[\mez, s-\mez-\delta]$.
Next we note that by using the variational form of \eqref{system:fpm} derived in  Proposition 4.8 \cite{NgPa} it can be proved that the following $\wt H^\mez_\pm$ contraction estimate holds
\bq
\| f^\pm\|_{\wt H^\mez_\pm}\le \cF(N_s)(\| k_\delta\|_{H^\mez}+\| \eta_\delta\|_{H^s}\|(k_1, k_2)\|_{H^\mez}).
\eq
By virtue of Theorem \ref{est:diffF(U)} and the embedding $\na \eta_j \in H^{s-1}\subset L^\infty$, we have
\begin{align*}
\| k_\delta\|_{H^\mez}&\le \s \Big\| \frac{\na\eta_1}{\sqrt{1+|\na\eta_1|^2}}-\frac{\na\eta_2}{\sqrt{1+|\na\eta_2|^2}}\Big\|_{H^\tdm}+ \lb \rho\rb\g\|\eta_\delta\|_{H^\mez}\\
&\le\s \cF(N_s)(\| \eta_\delta\|_{H^{\frac{5}{2}}}+\| \eta_\delta\|_{H^s})+ \lb \rho\rb\g\|\eta_\delta\|_{H^\mez}.
\end{align*}
It follows that
\bq\label{fd:low}
\begin{aligned}
\| f^\pm\|_{\wt H^\mez_\pm}&\le \cF(N_s)(\s\| \eta_\delta\|_{H^{\frac 52}}+\lb \rho\rb\g\| \eta_\delta\|_{H^{\mez}})\\
&\quad +\cF(N_s)\| \eta_\delta\|_{H^s}\big(\s (N_{\frac 52}+1)+\lb \rho\rb\g N_{\mez}\big).
\end{aligned}
\eq
Then combining \eqref{key:estfd}, \eqref{fd:low} and an induction argument we arrive at
\bq\label{contra:fd:100}
\begin{aligned}
\| f^\pm_\delta\|_{\wt H^r_\pm}&\le \cF(N_s)(\s\| \eta_\delta\|_{H^{r+2}}+\lb \rho\rb\g\| \eta_\delta\|_{H^r})\\
&\quad +\cF(N_s)\| \eta_\delta\|_{H^s}\big(\s (N_{s+\tdm}+1)+\lb \rho\rb\g N_{s-\mez}\big)
\end{aligned}
\eq
for all $r\in [\mez, s-\mez]$. This proves \eqref{contra:fpmd}. Finally, \eqref{Tld:f-eta}-\eqref{estTld:f-eta} follow from \eqref{Tld:f-eta:0}, \eqref{estF:fd} and \eqref{contra:fd:100}.
\end{proof}
\subsubsection{Proof of Theorem \ref{theo:contra:2p}}
From equation \eqref{eq:eta2p} we see that $\eta_\delta=\eta_1-\eta_2$ satisfies 
\[
\p_t \eta_\delta=-\frac{1}{\mu^-}G^-(\eta_1)f^-_\delta-\frac{1}{\mu^-}[G^-(\eta_1)-G^-(\eta_2)]f^-_2.
\]
According to Theorem \ref{contraction:DN},
\[
\| [G^-(\eta_1)-G^-(\eta_2)]f^-_2\|_{H^{s-\tdm}}\le (\s+\lb \rho\rb\g)\cF(N_s)\| \eta_\delta\|_{H^s}N_{s+\tdm}.
\]
Applying Theorem \ref{paralin:ABZ}  and the estimate \eqref{contra:fpmd} for $f^-_\delta$ (with $r=s-\mez-\delta$) yields $G^-(\eta_1)f^-_\delta=T_{\ld_1}f^-_\delta+R^-(\eta_1)f^-_\delta$ where 
\[
\begin{aligned}
\| R^-(\eta_1)f^-_\delta\|_{H^{s-\tdm}}&\le \cF(N_s)\| f^-_\delta\|_{\wt H^{s-\mez-\delta}_-}\\
&\les  (\s+\lb \rho\rb\g)\cF(N_s)(\| \eta_\delta\|_{H^{s+\tdm-\delta}}+\| \eta_\delta\|_{H^s}N_{s+\tdm}).
\end{aligned}
\]
Thus, for some $\cF$ depending only on $(h, s, \delta, \mu^\pm)$ we have
\[
\begin{aligned}
&\p_t \eta_\delta=-\frac{1}{\mu^-}T_{\ld_1}f^-_\delta+\cR_1,\\
&\| \cR_1\|_{H^{s-\tdm}}\le (\s+\lb \rho\rb\g)\cF(N_s)(\| \eta_\delta\|_{H^{s+\tdm-\delta}}+\| \eta_\delta\|_{H^s}N_{s+\tdm}).
\end{aligned}
\]
By virtue of  \eqref{Tld:f-eta}-\eqref{estTld:f-eta} with $\sigma=s-\mez-\delta$, 
\[
\begin{aligned}
&T_{\ld_1}f^-_\delta=\frac{\s\mu^-}{\mu^++\mu^-}T_{\ld_1}(H(\eta_1)-H(\eta_2))+\frac{\lb \rho\rb\g\mu^-}{\mu^++\mu^-}T_{\ld_1}\eta_\delta +g_-,\\
&\| g_-\|_{H^{s-\tdm}}\le (\s+\lb \rho\rb\g)\cF(N_s)(\| \eta_\delta\|_{H^{s+\tdm-\delta}}+\| \eta_\delta\|_{H^s}N_{s+\tdm}).
\end{aligned}
\]
Clearly,
\[
\| T_{\ld_1}\eta_\delta\|_{H^{s-\tdm}}\le \cF(N_s)\| \eta\|_{H^{s-\mez}}.
\]
Then in view of \eqref{TldHd} we deduce that
\bq
\begin{aligned}
&\p_t\eta_\delta=-\frac{\s}{\mu^++\mu^-}T_{\ld_1\ell_1}\eta_\delta-\cR_2,\\
&\| \cR_2\|_{H^{s-\tdm}}\le (\s+\lb \rho\rb\g)\cF(N_s)\big(\| \eta_\delta\|_{H^{s+\tdm-\delta}}+N_{s+\tdm}\| \eta_\delta\|_{H^s}\big),
\end{aligned}
\eq
where $\cF$ depends only on $(h, s, \delta,  \mu^\pm)$. This reduction is of the same form as \eqref{req:etad}-\eqref{est:wtcR1} in the proof of Theorem \ref{theo:contra}. Thus, we can conclude similarly. 
\subsection{Proof of Theorem \ref{theo:2p}}
Let  $\eta_0\in H^s$ be an initial datum satisfying $\dist (\eta_0, \Gamma^\pm)>2h>0$. For each $\eps\in (0, 1)$, let $\eta_\eps$ solve the ODE 
\bq
\p_t\eta_\eps=-\frac{1}{\mu^-}J_\eps \Big[G^-(J_\eps \eta_\eps)\big(J_\eps f^-_\eps\big)\Big],\quad \eta_\eps\vert_{t=0}=\eta_0,
\eq
where $f^\pm_\eps$ solve
\bq
\begin{cases}
 f^{-}_\eps-f^{+}_\eps=\s H(\eta_\eps)+\lb \rho\rb\g\eta_\eps,\\
\frac{1}{\mu^+}G^+(\eta_\eps)f^{+}_\eps-\frac{1}{\mu^-}G^-(\eta_\eps)f^{-}_\eps=0.
\end{cases}
\eq
Note that the solvability and regularity of $f^\pm_\eps$ are guaranteed by Propositions \ref{prop:fpm} and \ref{prop:estfpm}. Since  the a priori estimates in Proposition \ref{apriori:2p} and the contraction estimate in Theorem \ref{theo:contra:2p} remain true for $\eta_\eps$, the existence, uniqueness and stability of solutions to \eqref{eq:eta2p}-\eqref{system:fpm} can be deduced  as in  the proof of Theorem \ref{theo:1p}.
\appendix
\section{A review of paradifferential calculus}\label{appendix:para}
We provide a review of basic features of Bony's paradifferential calculus (see e.g. \cite{BCD, Bony, Hormander, MePise}). 
\begin{defi}\label{defi:para}
1. (Symbols) Given~$\rho\in [0, \infty)$ and~$m\in\Rr$,~$\Gamma_{\rho}^{m}(\Rr^d)$ denotes the space of
locally bounded functions~$a(x,\xi)$
on~$\Rr^d\times(\Rr^d\setminus 0)$,
which are~$C^\infty$ with respect to~$\xi$ for~$\xi\neq 0$ and
such that, for all~$\alpha\in\Nn^d$ and all~$\xi\neq 0$, the function
$x\mapsto \partial_\xi^\alpha a(x,\xi)$ belongs to~$W^{\rho,\infty}(\Rr^d)$ and there exists a constant
$C_\alpha$ such that,
\begin{equation*}
\forall |\xi|\ge \mez,\quad 
\Vert \partial_\xi^\alpha a(\cdot,\xi)\Vert_{W^{\rho,\infty}(\Rr^d)}\le C_\alpha
(1+|\xi|)^{m-|\alpha|}.
\end{equation*}
Let $a\in \Gamma_{\rho}^{m}(\Rr^d)$, we define the semi-norm
\begin{equation}\label{defi:norms}
M_{\rho}^{m}(a)= 
\sup_{|\alpha|\le 2(d+2) +\rho ~}\sup_{|\xi| \ge \mez~}
\Vert (1+|\xi|)^{|\alpha|-m}\partial_\xi^\alpha a(\cdot,\xi)\Vert_{W^{\rho,\infty}(\Rr^d)}.
\end{equation}
2. (Paradifferential operators) Given a symbol~$a$, we define
the paradifferential operator~$T_a$ by
\begin{equation}\label{eq.para}
\widehat{T_a u}(\xi)=(2\pi)^{-d}\int \chi(\xi-\eta,\eta)\widehat{a}(\xi-\eta,\eta)\Psi(\eta)\widehat{u}(\eta)
\, d\eta,
\end{equation}
where
$\widehat{a}(\theta,\xi)=\int e^{-ix\cdot\theta}a(x,\xi)\, dx$
is the Fourier transform of~$a$ with respect to the first variable; 
$\chi$ and~$\Psi$ are two fixed~$C^\infty$ functions such that:
\begin{equation}\label{cond.psi}
\Psi(\eta)=0\quad \text{for } |\eta|\le \frac{1}{5},\qquad
\Psi(\eta)=1\quad \text{for }|\eta|\geq \frac{1}{4},
\end{equation}
and~$\chi(\theta,\eta)$ 
satisfies, for~$0<\eps_1<\eps_2$ small enough,
$$
\chi(\theta,\eta)=1 \quad \text{if}\quad |\theta|\le \eps_1| \eta|,\qquad
\chi(\theta,\eta)=0 \quad \text{if}\quad |\theta|\geq \eps_2|\eta|,
$$
and such that
$$
\forall (\theta,\eta), \qquad | \partial_\theta^\alpha \partial_\eta^\beta \chi(\theta,\eta)|\le 
C_{\alpha,\beta}(1+| \eta|)^{-|\alpha|-|\beta|}.
$$
\end{defi}
\begin{rema}
The cut-off $\chi$ can be appropriately chosen so that when $a=a(x)$, the paradifferential operator $T_au$ becomes the usual paraproduct.
\end{rema}
\begin{defi}\label{defi:order}
Let~$m\in\Rr$.
An operator~$T$ is said to be of  order~$m$ if, for all~$\mu\in\Rr$,
it is bounded from~$H^{\mu}$ to~$H^{\mu-m}$.
\end{defi}
Symbolic calculus for paradifferential operators is summarized in the following theorem.
\begin{theo}\label{theo:sc}(Symbolic calculus)
Let~$m\in\Rr$ and~$\rho\ge 0$. \\
$(i)$ If~$a \in \Gamma^m_0(\Rr^d)$, then~$T_a$ is of order~$m$. 
Moreover, for all~$\mu\in\Rr$ there exists a constant~$K$ such that
\begin{equation}\label{esti:quant1}
\Vert T_a \Vert_{H^{\mu}\rightarrow H^{\mu-m}}\le K M_{0}^{m}(a).
\end{equation}
$(ii)$ If~$a\in \Gamma^{m}_{\rho}(\Rr^d), b\in \Gamma^{m'}_{\rho}(\Rr^d)$ then 
$T_a T_b -T_{ab}$ is of order~$m+m'-\rho$. 
where
\[
a\sharp b:=\sum_{|\alpha|<\rho}\frac{(-i)^{\alpha}}{\alpha !}\partial_{\xi}^{\alpha}a(x, \xi)\partial_x^{\alpha}b(x, \xi).
\] 
Moreover, for all~$\mu\in\Rr$ there exists a constant~$K$ such that
\begin{equation}\label{esti:quant2}
\begin{aligned}
\Vert T_a T_b  - T_{a  b} \Vert_{H^{\mu}\rightarrow H^{\mu-m-m'+\rho}}
&\le 
K (M_{\rho}^{m}(a)M_{0}^{m'}(b)+M_{0}^{m}(a)M_{\rho}^{m'}(b)).
\end{aligned}
\end{equation}
$(iii)$ Let~$a\in \Gamma^{m}_{\rho}(\Rr^d)$. Denote by 
$(T_a)^*$ the adjoint operator of~$T_a$ and by~$\overline{a}$ the complex conjugate of~$a$. Then $(T_a)^* -T_{\overline{a}}$ is of order~$m-\rho$  where
\[
a^*=\sum_{|\alpha|<\rho}\frac{1}{i^{|\alpha|}\alpha!}\partial_{\xi}^{\alpha}\partial_x^{\alpha}\overline{a}.
\] 
Moreover, for all~$\mu$ there exists a constant~$K$ such that
\begin{equation}\label{esti:quant3}
\Vert (T_a)^*   - T_{a^*}   \Vert_{H^{\mu}\rightarrow H^{\mu-m+\rho}}
\le 
K M_{\rho}^{m}(a).
\end{equation}
\end{theo}
\begin{rema}\label{rema:low}
In the definition \eqref{eq.para} of paradifferential operators, the cut-off $\Psi$ removes the low frequency part of $u$. Consequently, when $a\in \Gamma^m_0$  we have
\[
\Vert T_a u\Vert_{H^\sigma}\le CM_0^m(a)\Vert \nabla u\Vert_{H^{\sigma+m-1}}\equiv CM_0^m(a)\Vert  u\Vert_{H^{1, \sigma+m}}. 
\]
The same remark applies to Theorem A.4 (ii) and (iii). 
\end{rema}
Next we recall several useful product and paraproduct rules.
\begin{theo}\label{pproduct}
Let $s_0$, $s_1$ and $s_2$ be real numbers.
\begin{enumerate}
\item For any $s\in \Rr$, 
\bq\label{pp:Linfty}
\| T_a u\|_{H^s}\le C\| a\|_{L^\infty}\| u\|_{H^s}.
\eq
\item If
$s_0\le s_2$ and~$s_0 < s_1 +s_2 -\frac{d}{2}$, 
then
\begin{equation}\label{boundpara}
\Vert T_a u\Vert_{H^{s_0}}\le C \Vert a\Vert_{H^{s_1}}\Vert u\Vert_{H^{s_2}}.
\end{equation}
\item If $s_1+s_2> 0$, $s_0\le s_1$ and $s_0< s_1+s_2-\frac{d}{2}$ then  
\bq\label{paralin:product}
\Vert au - T_a u\Vert_{H^{s_0}}\le C \Vert a\Vert_{H^{s_1}}\Vert u\Vert_{H^{s_2}}.
\eq
\item \label{it.1} If $s_1+s_2> 0$, $s_0\le s_1$, $s_0\le s_2$ and  $s_0< s_1+s_2-\frac d2
$ then 
\begin{equation}\label{pr}
\Vert u_1 u_2 \Vert_{H^{s_0}}\le C \Vert u_1\Vert_{H^{s_1}}\Vert u_2\Vert_{H^{s_2}}.
\end{equation}
\end{enumerate}
\end{theo}
\begin{theo}[\protect{\cite[Theorem~2.89]{BCD}}]\label{est:nonl}
  Consider~$F\in C^\infty(\Cc^N)$ such that~$F(0)=0$.  For $s\ge 0$, there exists a non-decreasing function~$\mathcal{F}\colon\Rr_+\rightarrow\Rr_+$ 
such that, for any~$U\in H^s(\Rr^d)^N\cap L^\infty(\Rr^d)^N$,
\begin{equation}\label{est:F(u):S}
\Vert F(U)\Vert_{H^s}\le \mathcal{F}\bigl(\Vert U\Vert_{L^\infty}\bigr)\Vert U\Vert_{H^s}.
\end{equation}
\end{theo}
\begin{theo}[\protect{\cite[Corollary~2.90]{BCD}}]\label{est:diffF(U)}
Consider~$F\in C^\infty(\Cc^N)$ such that~$\na F(0)=0$.  For $s\ge 0$, there exists a non-decreasing function~$\mathcal{F}\colon\Rr_+\rightarrow\Rr_+$ 
such that, for any~$U,~V\in H^s(\Rr^d)^N\cap L^\infty(\Rr^d)^N$,
\begin{equation}
\Vert F(U)-F(V)\Vert_{H^s}\le \mathcal{F}\bigl(\Vert (U, V)\Vert_{L^\infty}\bigr)\Big(\Vert U-V\Vert_{H^s}+\| U-V\|_{L^\infty}\sup_{\tau \in [0, 1]}\|V+\tau(U-V)\|_{H^s} \Big).
\end{equation}
\end{theo}
\begin{theo}[\protect{\cite[Theorem~2.92]{BCD} and \cite[Theorem 5.2.4]{MePise}}]\label{paralin:nonl}(Paralinearization for nonlinear functions)
Let $\mu,~\tau$ be positive real numbers and let $F\in C^{\infty}(\mathbb{C}^N)$ be a scalar function satisfying $F(0)=0$. If $U=(u_j)_{j=1}^N$ with $u_j\in H^{\mu}(\Rr^d)\cap C_*^{\tau}(\Rr^d)$ then we have
\begin{align}\label{def:RF}
&F(U)=\Sigma_{j=1}^NT_{\p_jF(U)}u_j+R,\\
&\| R\|_{ H^{\mu+\tau}}\le \cF(\| U\|_{L^{\infty}})\| U\|_{C_*^{\tau}}\| U\|_{H^{\mu}}.
\end{align}
\end{theo}
\section{Proof of Proposition \ref{prop:reform}}\label{appendix:reform}
Setting $q^\pm=p^\pm+\rho^\pm \g y$ we deduce from the Darcy law \eqref{Darcy:pm} that  
\bq\label{Darcy:uq}
\Delta_{x, y}q^\pm=0,\quad \mu^\pm u^\pm=-\na_{x, y}q^\pm \quad\text{in}\quad \Omega^\pm.
\eq
 {\it The one-phase problem.}  Then boundary condition \eqref{bcp:1p} gives  $q^-\vert_\Sigma= \s H(\eta)+\lb \rho\rb\g \eta$. Consequently, by the definition of $G^-(\eta)$ we have
\[
\sqrt{1+|\na \eta|^2}\na_{x, y}q^-\cdot n\vert_\Sigma=G^-(\eta)(\s H(\eta)+\rho^-\g \eta)
\]
which in conjunction with \eqref{Darcy:uq} yields 
\[
\sqrt{1+|\na \eta|^2}u^-\cdot n\vert_\Sigma=-\frac{1}{\mu^-}G^-(\eta)(\s H(\eta)+\rho^-\g \eta).
\]
Combing this and the kinematic boundary condition \eqref{kbc:pm} we obtain equation \eqref{eq:eta}.

{\it The two-phase problem.} Set $f^\pm=q^\pm\vert_\Sigma=p^\pm\vert_\Sigma+\rho^\pm\g \eta$. In view of the pressure jump condition \eqref{p:pm} we have 
\[
f^--f^+=\s H(\eta)+\lb \rho \rb\g\eta,\quad \lb \rho\rb=\rho^--\rho^+
\] 
which gives the first equation in \eqref{system:fpm}. On the other hand, since
\[
G^\pm(\eta)f^\pm=\sqrt{1+|\na \eta|^2}\na_{x, y}q^\pm\cdot n\vert_\Sigma,
\]
\eqref{Darcy:uq} implies that 
\bq\label{reform:100}
\mu^\pm \sqrt{1+|\na \eta|^2} u^\pm\cdot n\vert_\Sigma=-G^\pm(\eta)f^\pm.
\eq
The second equation in \eqref{system:fpm} thus follows from \eqref{reform:100} and the continuity \eqref{u.n:pm} of $u\cdot n$. Finally, \eqref{eq:eta2p} is a consequence of \eqref{kbc:pm} and \eqref{reform:100}.
\section{Estimates for paradifferential symbols}
We prove  estimates for the symbols defined in terms of $\ld$ and $\ell$ (see \eqref{ld} and \eqref{def:ell}) that are used in the proof of the main results. 
\begin{lemm}\label{lemm:symbol}
Let $s>1+\frac{d}{2}$ and $\delta \in (0, s-1-\frac d2)$. Then, there exists $\cF:\Rr^+\to \Rr^+$ such that 
\begin{align}\label{symbolest:1}
&M^1_\delta(\ld)+M^2_\delta(\ell)\le \cF(\| \na\eta\|_{H^{s-1}(\Rr^d)}),\\ \label{symbolest:2}
&M^3_\delta(\ld\ell)+M^\tdm_\delta(\sqrt{\ld \ell})+M^{-\tdm}_\delta(\sqrt{\ld \ell}^{-1}) \le\cF(\| \na\eta\|_{H^{s-1}(\Rr^d)}),\\ \label{symbolest:3}
&M_\delta^2(\ell_1-\ell_2)\le \cF(\|(\na\eta_1, \na\eta_2)\|_{H^{s-1}(\Rr^d)})\| \na\eta_1-\na\eta_2\|_{H^{s-1}(\Rr^d)}
\end{align}
for all $\eta,~\eta_1,~\eta_2\in H^{1, s}(\Rr^d)$ (see definition \eqref{def:dotHs}).
\end{lemm}
\begin{proof}
To prove \eqref{symbolest:1} for $\ld$, we rewrite $\ld$ as 
\[
\ld(x, \xi)=\big((1+|\nabla\eta(x)|^2)-(\nabla\eta(x)\cdot \frac{\xi}{|\xi|})^2\big)^\mez|\xi|:=g(\na \eta(x), \xi)|\xi|.
\]
Note that $g(0, \xi)=1$ and $g(\na \eta(x), \xi)\ge 1$ for all $(x, \xi)\in \Rr^d\times \Rr^d$. Applying Theorem \ref{est:nonl} and the Sobolev embedding $H^{s-1}(\Rr^d)\subset W^{\delta, \infty}(\Rr^d)$ we get 
\[
\| g(\na \eta, \xi)\|_{W^{\delta, \infty}(\Rr^d)}\le C\| g(\na \eta, \xi)-1\|_{H^{s-1}(\Rr^d)}+1\le \cF(\| \na \eta\|_{H^{s-1}(\Rr^d)}).
\]
It follows that 
\[
\| \partial_\xi^\alpha a(\cdot,\xi)\Vert_{W^{\delta,\infty}(\Rr^d)}\le C_\alpha\cF(\| \na \eta\|_{H^{s-1}(\Rr^d)})(1+| \xi|)^{1-|\alpha|}
\]
for all $\alpha\in \Nn^d$ and $|\xi|\ge \mez$. In view of the definition \eqref{defi:norms} of $M^m_\rho$, we obtain \eqref{symbolest:1}. As for $\ell$, we rewrite \eqref{def:ell} as
\[
\ell(x, \xi)=(1+|\na\eta|^2)^{-\tdm}\big((1+|\na \eta|^2)-(\na \eta\cdot \frac{\xi}{|\xi|})^2\big)|\xi|^2:=F(\na \eta, \xi)|\xi|^2
\]
and argue similarly. Note that $F(0, \xi)=1$ and the gradient of $F$ with respect to the first argument vanishes at $(0, \xi)$. This finishes the proof of \eqref{symbolest:1}. Since $\ld\ell=(1+|\na \eta|^2)^{-\tdm}\ld^3$, the estimates in \eqref{symbolest:2}  follow from \eqref{symbolest:1} for $\ld$, the chain rule and calculus inequalities.

Regarding \eqref{symbolest:3}, we apply Theorem \ref{est:diffF(U)} and the embedding $H^{s-1}(\Rr^d)\subset L^\infty(\Rr^d)$ to have
\begin{align*}
\| F(\na \eta_1(\cdot), \xi)-F(\na \eta_2(\cdot), \xi)\|_{W^{\delta, \infty}}&\le \| F(\na \eta_1(\cdot), \xi)-F(\na \eta_2(\cdot), \xi)\|_{H^{s-1}}\\
&\le\cF(\| (\na \eta_1, \na\eta_2)\|_{H^{s-1}})\| \na \eta_1-\na\eta_2\|_{H^{s-1}}
\end{align*}
for all $\xi\in \Rr^d$. Then, \eqref{symbolest:3} follows as above.
\end{proof}

\vspace{.1in}
\noindent{\bf{Acknowledgment.}} 
The work of HQN was partially supported by NSF grant DMS-1907776. The author thanks B. Pausader and F. Gancedo for many discussions on the Muskat problem. We would like to thank the reviewer for his/her positive comment and detailed reading.


\end{document}